\documentclass[11pt]{amsart}
\usepackage{amssymb}
\usepackage{amsthm}
\usepackage{amsmath}
\usepackage{latexsym}
\usepackage{comment}
\usepackage{hyperref}
\usepackage{verbatim}
\usepackage{xypic}
\usepackage{graphicx}
\usepackage[utf8]{inputenc}
\usepackage[margin=1in]{geometry}
\graphicspath{ {pics/} }
\newtheorem{thm}{Theorem}[section]
\newtheorem{prop}[thm]{Proposition}

\newtheorem{lemma}[thm]{Lemma}
\newtheorem{cor}[thm]{Corollary}
\newtheorem{dfn}[thm]{Definition}

\theoremstyle{definition} \newtheorem{ex}[thm]{Example}

\theoremstyle{definition} \newtheorem{rmk}[thm]{Remark}
\newcommand{\cc}{\mathbb{C}}

\newcommand{\qq}{\mathbb{Q}}
\newcommand{\zz}{\mathbb{Z}}
\newcommand{\ff}{\mathbb{F}}

\newcommand{\proj}{\mathbb{P}}

\newcommand{\Gal}{\mathrm{Gal}}

\newcommand{\Sp}{\mathrm{Sp}}
\newcommand{\GSp}{\mathrm{GSp}}
\newcommand{\Hom}{\mathrm{Hom}}
\newcommand{\End}{\mathrm{End}}
\newcommand{\Aut}{\mathrm{Aut}}
\newcommand{\Spec}{\mathrm{Spec}}
\newcommand{\Div}{\mathrm{Div}}
\newcommand{\Pic}{\mathrm{Pic}}

\newcommand{\et}{\mathrm{\acute{e}t}}
\newcommand{\tame}{\mathrm{tame}}
\newcommand{\tp}{\mathrm{top}}
\newcommand{\alg}{\mathrm{alg}}

\graphicspath{{pics/}}

\title{Boundedness results for $2$-adic Galois images associated to hyperelliptic Jacobians}
\author{Jeffrey Yelton}

\begin{document}

\maketitle

\begin{abstract}

Let $K$ be a number field, and let $C$ be a hyperelliptic curve over $K$ with Jacobian $J$.  Suppose that $C$ is defined by an equation of the form $y^{2} = f(x)(x - \lambda)$ for some irreducible monic polynomial $f \in \mathcal{O}_{K}[x]$ of discriminant $\Delta$ and some element $\lambda \in \mathcal{O}_{K}$.  Our first main result says that if there is a prime $\mathfrak{p}$ of $K$ dividing $(f(\lambda))$ but not $(2\Delta)$, then the image of the natural $2$-adic Galois representation is open in $\GSp(T_{2}(J))$ and contains a certain congruence subgroup of $\Sp(T_{2}(J))$ depending on the maximal power of $\mathfrak{p}$ dividing $(f(\lambda))$.  We also present and prove a variant of this result that applies when $C$ is defined by an equation of the form $y^{2} = f(x)(x - \lambda)(x - \lambda')$ for distinct elements $\lambda, \lambda' \in K$.  We then show that the hypothesis in the former statement holds for almost all $\lambda \in \mathcal{O}_{K}$ and prove a quantitative form of a uniform boundedness result of Cadoret and Tamagawa.

\end{abstract}

\section{Introduction} \label{S1}

Let $K$ be a number field with absolute Galois group $G_{K}$, and let $C$ be a hyperelliptic curve defined over $K$; i.e. $C$ is a smooth projective curve defined by an equation of the form $y^{2} = f(x)$ for some squarefree polynomial $f$ of degree $d \geq 3$.  (Note that in the case of $d = 3$, $C$ is an elliptic curve.)  It is well known that the genus of $C$ is given by $g = \lfloor (d + 1) / 2 \rfloor$.  We denote the Jacobian variety of $C$ by $J$; it is an abelian variety of dimension $g$.  For each prime $\ell$, we let $T_{\ell}(J)$ denote the $\ell$-adic Tate module of $J$, which is a free $\zz_{\ell}$-module of rank $2g$.  We write $\rho_{\ell} : G_{K} \to \Aut(T_{\ell}(J))$ for the natural $\ell$-adic Galois action on this Tate module.  The Tate module $T_{\ell}(J)$ is endowed with the Weil pairing defined with respect to the canonical principal polarization on $J$, which we write as $e_{\ell} : T_{\ell}(J) \times T_{\ell}(J) \to \zz_{\ell}$; it is a $\zz_{\ell}$-bilinear skew-symmetric pairing.  Let $\Sp(T_{\ell}(J))$ denote the group of symplectic automorphisms of $T_{\ell}(J)$ with respect to the pairing $e_{\ell}$, and let 
$$\GSp(T_{\ell}(J)) := \{\sigma \in \mathrm{Aut}_{\zz_{\ell}}(T_{\ell}(J))\ |\ e_{\ell}(P^{\sigma}, Q^{\sigma}) = \chi_{\ell}(\sigma) e_{\ell}(P, Q) \ \forall P, Q \in T_{\ell}(J)\}$$
 denote the group of symplectic similitudes, where $\displaystyle \chi_{\ell} : G_{K} \to \zz_{\ell}^{\times}$ is the $\ell$-adic cyclotomic character.

It is well known that the image $G_{\ell}$ of $\rho_{\ell}$ is always a closed subgroup of $\GSp(T_{\ell}(J))$ and that in fact there is some hyperelliptic Jacobian $J$ of a given dimension $g$ such that the inclusion $G_{\ell} \subseteq \GSp(T_{\ell}(J))$ has finite index (or equivalently, that $G_{\ell}$ is an open subgroup of the $\ell$-adic Lie group $\GSp(T_{\ell}(J))$); see for instance \cite[Theorem 1.1]{yelton2015images}.  Note that the subgroup $G_{\ell} \cap \Sp(T_{\ell}(J)) \subset G_{\ell}$ coincides with the image of the Galois subgroup which fixes the extension obtained by adjoining all $\ell$-power roots of unity to $K$.  By \cite[Theorem 3]{bogomolov1981points}, the Lie algebra of $G_{\ell}$ contains the subalgebra of homotheties in the Lie algebra of $\GSp(T_{\ell}(J))$; it follows that $G_{\ell}$ has finite index in $\GSp(T_{\ell}(J))$ if and only if $G_{\ell} \cap \Sp(T_{\ell}(J))$ has finite index in $\Sp(T_{\ell}(J))$.

There have been many results stating that $G_{\ell}$ has finite index in $\GSp(T_{\ell}(J))$ under various hypotheses for the polynomial defining the hyperelliptic curve.  For instance, Y. Zarhin has proven this for large enough genus in the case of hyperelliptic curves defined by equations of the form $y^{2} = f(x)$ or $y^{2} = f(x)(x - \lambda)$ with $\lambda \in K$, where the Galois group of $f$ is the full symmetric or alternating group (\cite[Theorem 2.5]{zarhin2002very} and \cite[Theorem 8.3]{zarhin2010families}; see also \cite[Theorem 1.3]{zarhin2013two} for a variant of this where the curve is defined using two parameters).  A. Cadoret and A. Tamagawa have also proven (\cite[Theorems 1.1 and 5.1]{cadoret2012uniform}) that for any family of hyperelliptic Jacobians over a smooth, geometrically connected, separated curve over $K$, this openness condition will be satisfied for the $\ell$-adic Galois action associated to the fibers over all but finitely many of the $K$-points of the base, and that in fact the indices of the $\ell$-adic Galois images corresponding to these fibers are uniformly bounded.  However, there have been very few results which give explicit bounds for the index of $G_{\ell}$ in $\GSp(T_{\ell}(J))$ in such cases.

Our aim in this paper is to give some similar results on the openness of the $\ell$-adic Galois images in the group of symplectic similitudes associated to Jacobians of hyperelliptic curves whose defining polynomials satisfy certain hypotheses, and to provide formulas giving explicit bounds for the indices of the $2$-adic Galois images in these cases.  (Unfortunately, our method currently cannot give us explicit bounds for the indices of the $\ell$-adic Galois images for odd primes $\ell$.)

We state our main results below.  In these statements as well as in the rest of the paper, we use the following notation.  For any integer $N \geq 1$, we denote the level-$N$ congruence subgroup of $\Sp(T_{2}(J))$ by $\Gamma(N) := \{\sigma \in \Sp(T_{2}(J)) \ | \ \sigma \equiv 1 \ (\mathrm{mod} \ N)\}$.  We denote the ring of integers of a number field $K$ by $\mathcal{O}_{K}$. Finally, we write $v_{2} : \qq^{\times} \to \zz$ for the (normalized) $2$-adic valuation on $\qq$.

\begin{thm} \label{thm main1}

Let $K$ be a number field, and let $f \in \mathcal{O}_{K}[x]$ be an irreducible monic polynomial of degree $d \geq 2$ with discriminant $\Delta$.  Let $J$ be the Jacobian of the hyperelliptic curve with defining equation $y^{2} = f(x)(x - \lambda)$ for some $\lambda \in \mathcal{O}_{K}$, and define the $\ell$-adic Galois images $G_{\ell}$ as above.  Suppose that there is a prime $\mathfrak{p}$ of $\mathcal{O}_{K}$ which divides $(f(\lambda))$ but not $(2\Delta)$.

a) The Lie subgroup $G_{\ell} \subset \GSp(T_{\ell}(J))$ is open for all $\ell$ different from the residue characteristic of $\mathfrak{p}$.  In particular, the Lie subgroup $G_{2} \subset \GSp(T_{2}(J))$ is open.

b) We have $G_{2} \cap \Sp(T_{2}(J)) \supsetneq \Gamma(2^{2v_{2}(m) + 2})$, where $m \geq 1$ is the greatest integer such that $\mathfrak{p}^{m} \mid (f(\lambda))$.  If in addition $d = 3$, then $G_{2} \cap \Sp(T_{2}(J)) \supsetneq \Gamma(2^{v_{2}(m) + 1})$.

\end{thm}

\begin{thm} \label{thm main2}

Assume the same notation as in the statement of Theorem \ref{thm main1}, except that the defining equation of the hyperelliptic curve is $y^{2} = f(x)(x - \lambda)(x - \lambda')$ for distinct elements $\lambda, \lambda' \in \mathcal{O}_{K}$.  Suppose that there is a prime $\mathfrak{p}$ of $\mathcal{O}_{K}$ which divides $(f(\lambda))$ but not $(\lambda - \lambda')$ or $(2\Delta)$ and a prime $\mathfrak{p}'$ which divides $(\lambda - \lambda')$ but not $(f(\lambda))$ or $(2\Delta)$.

a) The Lie subgroup $G_{\ell} \subset \GSp(T_{2}(J))$ is open for all $\ell$ different from the residue characteristics of $\mathfrak{p}$ and $\mathfrak{p'}$.  In particular, the Lie subgroup $G_{2} \subset \GSp(T_{2}(J))$ is open.

b) We have $G_{2} \cap \Sp(T_{2}(J)) \supsetneq \Gamma(2^{2v_{2}(m) + 2})$ if $v_{2}(m') \leq v_{2}(m)$ or $d = 2g$ and $G_{2} \cap \Sp(T_{2}(J)) \supsetneq \Gamma(2^{v_{2}(m) + v_{2}(m') + 2})$ otherwise, where $m \geq 1$ is the greatest integer such that $\mathfrak{p}^{m} \mid (f(\lambda))$ and $m' \geq 1$ is the greatest integer such that $\mathfrak{p}'^{m'} \mid (\lambda - \lambda')$.  If in addition $d = 2$, then $G_{2} \cap \Sp(T_{2}(J)) \supsetneq \Gamma(2^{\max \{v_{2}(m), v_{2}(m')\} + 1})$.

\end{thm}

\begin{rmk} \label{rmk almost all fibers}

For a fixed irreducible monic polynomial $f \in \mathcal{O}_{K}[x]$ of degree $d \geq 2$ with discriminant $\Delta$, it is not hard to show using Faltings' Theorem that the hypothesis in Theorems \ref{thm main1} is satisfied for almost all $\lambda \in \mathcal{O}_{K}$ (in \S\ref{S5}, we will prove in a similar fashion that it is satisfied for almost all $\lambda \in K$ if $d \geq 4$).  It follows immediately that there are also infinitely many choices of $(\lambda, \lambda') \in K \times K$ satisfying the hypotheses in Theorem \ref{thm main2} for this polynomial $f$.

\end{rmk}

\begin{cor} \label{cor end}

If $J$ is the Jacobian of a hyperelliptic curve satisfying the hypotheses of Theorem \ref{thm main1} or \ref{thm main2}, then its endomorphism ring $\End(J)$ coincides with $\zz$.  In particular, such a Jacobian $J$ is absolutely simple.

\end{cor}

\begin{proof}

It is known that there is a finite algebraic extension $K'$ of $K$ over which every endomorphism of an abelian variety over a field $K$ is defined (see \cite[Theorem 2.4]{silverberg1992fields}), so that each endomorphism in $\End(J)$ commutes with the action of $\Gal(\bar{K} / K')$ on torsion points.  The corollary now follows from noting that the only endomorphisms in $\End(T_{\ell}(J))$ which commute with everything in an open subgroup of $\GSp(T_{\ell}(J))$ are scalars.

\end{proof}

The key ingredient used in proving the above theorems is a method of describing Galois actions on $\ell$-adic Tate modules of hyperelliptic Jacobians defined over the fraction field of a strictly Henselian DVR of residue characteristic $p \neq 2, \ell$ by looking at the valuations of the differences between the roots of the defining polynomial, which is derived from results shown in joint work with H. Hasson (\cite{hasson2020prime}).  This way of looking at $\ell$-adic Galois actions associated to hyperelliptic Jacobians over local fields is very similar to the ``method of clusters" used by S. Anni and V. Dokchitser in \cite{anni2020constructing}.  Our approach seems quite powerful and should lead to many similar boundedness results in a number of situations where one can compute valuations of the differences between the roots of the defining polynomial with respect to various primes of the ground field.  Unfortunately, in practice, these valuations (or even the roots themselves) may be difficult to calculate, and so our main focus here is on obtaining results such as the ones stated above where the hypotheses are very easy to verify.

The rest of this paper is organized as follows.  In \S\ref{S2}, we use the main results of \cite{hasson2020prime}, which show that over the fraction field of a strictly Henselian DVR of characteristic $p \neq 2$, for primes $\ell \neq p$, the $\ell$-adic Galois action factors through the tame quotient of the absolute Galois group and can be described in terms of Dehn twists with respect to certain loops on a complex hyperelliptic curve.  In particular cases such as when exactly two roots of the defining polynomial coalesce in the reduction over the residue field, we will show (Proposition \ref{prop Galois action local}) that such a Dehn twist induces a transvection in the symplectic group.  We will later put local data together to show that over a number field $K$, the $\ell$-adic Galois image contains certain powers of several sufficiently ``independent" transvections.  In \S\ref{S3}, we will demonstrate using elementary matrix algebra that the group generated by these powers of transvections is ``large" and, in the $\ell = 2$ case, contains a certain congruence subgroup.  In \S\ref{S4}, we will use what we have shown in \S\ref{S2} and \S\ref{S3} to prove Theorems \ref{thm main1} and \ref{thm main2} as well as to prove an auxiliary result that applies to a more general situation (Theorem \ref{thm several primes}).  In \S\ref{S5}, we will show using Theorem \ref{thm main1} that for a given $f \in \mathcal{O}_{K}[x]$ of degree $d \geq 3$, the $2$-adic Galois image associated to the hyperelliptic curves defined by $y^{2} = f(x)(x - \lambda)$ for all but finitely many $\lambda \in \mathcal{O}_{K}$ contains a principal congruence subgroup which depends only on $d$ (Theorem \ref{thm uniform bounds}).  In fact, for $d \geq 4$ even, in Theorem \ref{thm uniform bounds}(c) we will provide a uniform bound for indices of the $2$-adic Galois images associated to almost all fibers of such a one-parameter family over the $K$-line, as is guaranteed by \cite[Theorems 1.1 and 5.1]{cadoret2012uniform}.  Finally, in \S\ref{S6}, we will describe an alternate approach to proving the key Proposition \ref{prop Galois action local} from \S\ref{S2} in the case that $\ell = 2$.

The author is grateful to a MathOverflow user whose comment on question 264281 helped to inspire the arguments for the results in \S\ref{S5}.  The author would also like to thank the referee of an earlier version of this article, for a suggestion which has allowed him to remove an earlier hypothesis on the class number of $K$ from the statements of the results in \S\ref{S5}.  The author would also like to thank another referee who provided an enlightening strategy for the proof of (an alternate formulation of) Proposition \ref{prop Galois action local}, which is now described in \S\ref{S6}.

\section{Hyperelliptic Jacobians over local fields and tame Galois actions} \label{S2}

We retain all notation introduced in the previous section.  In this section, we write $\widehat{\zz}$ for the profinite completion of $\zz$ and use the symbol $\widehat{\pi}_{1}$ to denote the profinite completion of the fundamental group of a topological space.  For any profinite group $G$, we write $G^{(p')}$ for its maximal prime-to-$p$ quotient.  Note that since $G^{(p')}$ is a characteristic quotient of $G$, any action on $G$ induces an action on $G^{(p')}$.

Now we choose a prime $\mathfrak{p}$ of $K$ of residue characteristic $p \neq 2$.  Fix a strict Henselization of the localization of $K$ at the prime $\mathfrak{p}$ and denote it by $\mathcal{R}_{\mathfrak{p}}$ and its fraction field by $\mathcal{K}_{\mathfrak{p}}$; this comes with an embedding $\mathcal{K}_{\mathfrak{p}} \hookrightarrow \bar{K}$.  Let $\pi \in K$ be a uniformizer of the discrete valuation ring $\mathcal{R}_{\mathfrak{p}}$.  We fix a compatible system of $N$th roots of unity $\zeta_{N} \in \bar{K}$ for $N = 1, 2, 3, ...$; that is, we require that $\zeta_{N'N}^{N'} = \zeta_{N}$ for any integers $N, N' \geq 1$.  Note that since $R$ is strictly Henselian, $\zeta_{N} \in R \subset K$ for any $N$ not divisible by $p$.  Let $G_{K, \mathfrak{p}}$ denote the absolute Galois group of $\mathcal{K}_{\mathfrak{p}}$, and let $G_{K, \mathfrak{p}}^{\tame}$ denote its tame quotient.  It follows from a special case of Abhyankar's Lemma that the maximal tamely ramified extension $\mathcal{K}_{\mathfrak{p}}^{\tame}$ is given by $\mathcal{K}_{\mathfrak{p}}(\{\pi^{1/N}\}_{(N, p) = 1})$, where $\pi^{1/N}$ denotes an $N$th root of $\pi$, and that $G_{K, \mathfrak{p}}^{\tame} \cong \widehat{\zz}^{(p')}$ is topologically generated by the automorphism which acts on $K^{\tame}$ by fixing $K$ and sending each $\pi^{1/N}$ to $\zeta_{N}\pi^{1/N}$.

We fix, once and for all, an embedding $\bar{K} \hookrightarrow \cc$ where $\zeta_{N}$ is sent to $e^{2 \pi \sqrt{-1} / N}$ for $N \geq 1$, so that we have an inclusion $\mathcal{K}_{\mathfrak{p}} \subset \cc$.  Let $d \geq 2$ be an integer and choose distinct integral elements $\alpha_{1}, ... , \alpha_{d} \in K$.  Choose polynomials $\tilde{\alpha}_{1}, ... , \tilde{\alpha}_{d} \in \cc[x]$ satisfying $\tilde{\alpha}_{i}(\pi) = \alpha_{i}$ for $1 \leq i \leq d$ and such that the $x$-adic valuation of $\tilde{\alpha}_{i}$ and $\tilde{\alpha}_{j}$ and the $\pi$-adic valuation of $\alpha_{i}$ and $\alpha_{j}$ are equal (such polynomials exist, as is shown in the discussion in \cite[\S3.3]{hasson2020prime}).  Let $\varepsilon > 0$ be a real number small enough that $\tilde{\alpha}_{i}(z) \neq \tilde{\alpha}_{j}(z)$ for all $i \neq j$ and for all $z \in B_{\varepsilon}^{*} := \{z \in \cc \ | \ |z| < \varepsilon\} \smallsetminus \{0\}$.  We define a family $\mathcal{F} \to B_{\varepsilon}^{*}$ of $d$-times-punctured Riemann spheres by letting 
$$\mathcal{F} = \proj_{\cc}^{1} \times B_{\varepsilon}^{*} \smallsetminus \bigcup_{i = 1}^{d} \{(\tilde{\alpha}_{i}(z), z) \ | \ z \in B_{\varepsilon}^{*}\}.$$
Choose a basepoint $z_{0} \in B_{\varepsilon}^{*}$.  The fundamental group $\pi_{1}(B_{\varepsilon}^{*}, z_{0})$ acts by monodromy on the fundamental group $\pi_{1}(\mathcal{F}_{z_{0}}, \infty)$ of the fiber over $z_{0}$ with basepoint $\infty$.  We write $\rho_{\tp} : \pi_{1}(B_{\varepsilon}^{*}, z_{0}) \to \Aut(\pi_{1}(\mathcal{F}_{z_{0}}, \infty))$ for this action.  The action $\rho_{\tp}$ extends uniquely to a continuous action of the profinite completion $\widehat{\pi}_{1}(B_{\varepsilon}^{*}, z_{0})$ on the profinite completion $\widehat{\pi}_{1}(\mathcal{F}_{z_{0}}, \infty)$ (see the discussion in \cite[\S1.1]{hasson2020prime}), which we also denote by $\rho_{\tp}$.  We note that $\widehat{\pi}_{1}(B_{\varepsilon}^{*}, z_{0})$ is isomorphic to $\widehat{\zz}$ and is topologically generated by the element $\delta \in \pi_{1}(B_{\varepsilon}^{*}, z_{0})$ represented by the loop given by $t \mapsto e^{2 \pi \sqrt{-1} t}z_{0}$ for $t \in [0, 1]$.

Meanwhile, the absolute Galois group $G_{K, \mathfrak{p}}$ acts naturally by conjugation on the \'{e}tale fundamental group $\pi_{1}^{\et}(\proj_{\bar{K}_{\mathfrak{p}}}^{1} \smallsetminus \{\alpha_{1}, ... , \alpha_{d}\}, \infty)$ via the $\mathcal{K}_{\mathfrak{p}}$-point lying under the geometric point $\infty : \Spec(\cc) \to \proj_{\mathcal{K}_{\mathfrak{p}}}^{1} \smallsetminus \{\alpha_{1}, ... , \alpha_{d}\}$.  After identifying $\pi_{1}^{\et}(\proj_{\bar{K}}^{1} \smallsetminus \{\alpha_{1}, ... , \alpha_{d}\}, \infty)^{(p')}$ with $\widehat{\pi}_{1}(\proj_{\cc}^{1} \smallsetminus \{\alpha_{1}, ... , \alpha_{d}\}, \infty)^{(p')}$ via Riemann's Existence Theorem and the inclusion of algebraically closed fields $\bar{\mathcal{K}}_{\mathfrak{p}} \subset \cc$, we write $\rho_{\alg} : G_{K, \mathfrak{p}} \to \Aut(\widehat{\pi}_{1}(\proj_{\cc}^{1} \smallsetminus \{\alpha_{1}, ... , \alpha_{d}\}, \infty)^{(p')})$ for this action.  We denote the actions on prime-to-$p$ quotients of \'{e}tale fundamental groups induced by $\rho_{\tp}$ and $\rho_{\alg}$ by $\rho_{\tp}^{(p')}$ and $\rho_{\alg}^{(p')}$ respectively.

For the statement of Theorem \ref{thm comparison punctured projective line}(a) below, we require the terminology of Dehn twists.  Let $\gamma : [0, 1] \to M$ be a simple loop on any complex manifold $M$; we will often identify $\gamma$ with its image in $M$.  We define the \textit{Dehn twist} on $M$ with respect to the loop $\gamma$.  It is an element of the mapping class group of $M$ represented by a self-homeomorphism of $M$ which can be visualized in terms of a small tubular neighborhood of $\gamma \subset M$, in the following way: the Dehn twist keeps the outer edge of the tubular neighborhood fixed while twisting the inner edge one full rotation counterclockwise and acts as the identity everywhere else on $M$.  Since this Dehn twist depends only on the homology class $[\gamma] \in H_{1}(M, \zz)$ of any loop $\gamma$, we will denote it by $D_{[\gamma]}$.  (See \cite[Chapter 3]{farb2011primer} for more details.)

The following theorem is a compilation of all the necessary results describing and comparing $\rho_{\tp}$ and $\rho_{\alg}$ that are proven in \cite{hasson2020prime} (Theorems 1.2 and 2.4 and Remark 4.11 of that paper).

\begin{thm} \label{thm comparison punctured projective line}

In the above situation, we have the following.

a) Let $\mathcal{I}$ be the set of all pairs $(I, n)$ where $I \subseteq \{1, ... , d\}$ is a subset and $n \geq 1$ is an integer such that $x^{n} \mid \tilde{\alpha}_{i} - \tilde{\alpha}_{j} \in \cc[x]$ for all $i, j \in I$ and such that $I$ is maximal among intervals with this property.  If $\varepsilon$ is small enough, there exist pairwise nonintersecting loops $\gamma_{I, n} : [0, 1] \to \mathcal{F}_{z_{0}} \smallsetminus \{\infty\}$ for each $(I, n) \in \mathcal{I}$ such that $\delta \in \pi_{1}(B_{\varepsilon}^{*}, z_{0})$ acts on $\pi_{1}(\mathcal{F}_{z_{0}}, \infty)$ in the same way that the product $\prod_{(I, d) \in \mathcal{I}} D_{[\gamma_{I, n}]}$ of Dehn twists on $\mathcal{F}_{z_{0}} \smallsetminus \{\infty\}$ does.  These loops $\gamma_{I, n}$ each have the property of separating the subset $\{\tilde{\alpha}_{i}(z_{0})\}_{i \in I}$ from its complement in $\{\tilde{\alpha}_{j}(z_{0})\}_{j = 1}^{d} \cup \{\infty\}$, and two such loops $\gamma_{I, n}$ and $\gamma_{I', n'}$ are homologous if and only if $I = I'$.

b) The actions $\rho_{\tp}^{(p')}$ and $\rho_{\alg}^{(p')}$ factor through $\pi_{1}(B_{\varepsilon}^{*}, z_{0})^{(p')}$ and $G_{K, \mathfrak{p}}^{\tame}$ respectively.

c) We have isomorphisms $\widehat{\pi}_{1}(B_{\varepsilon}^{*}, z_{0})^{(p')} \stackrel{\sim}{\to} G_{K, \mathfrak{p}}^{\tame}$ and $\phi : \widehat{\pi}_{1}(\mathcal{F}_{z_{0}}, \infty)^{(p')} \stackrel{\sim}{\to} \widehat{\pi}_{1}(\proj_{\cc}^{1} \smallsetminus \{\alpha_{1}, ... , \alpha_{d}\}, \infty)^{(p')}$ inducing an isomorphism of the actions $\rho_{\tp}^{(p')}$ and $\rho_{\alg}^{(p')}$.  Moreover, we can choose the isomorphism $\phi$ so that it takes any element represented by a loop on $\mathcal{F}_{z_{0}}$ separating some singleton $\{\tilde{\alpha}_{i}(z_{0})\}$ from its complement in $\{\tilde{\alpha}_{j}(z_{0})\}_{j = 1}^{d} \cup \{\infty\}$ to an element represented by a loop on $\proj_{\cc}^{1} \smallsetminus \{\alpha_{1}, ... , \alpha_{d}\}$ separating the singleton $\{\alpha_{i}\}$ from its complement in $\{\alpha_{j}\}_{j = 1}^{d} \cup \{\infty\}$.

\end{thm}

In other words, the prime-to-$p$ monodromy action $\rho_{\tp}^{(p')}$ can be described in terms of Dehn twists with respect to loops surrounding certain subsets of the removed points which coalesce at a certain rate as one approaches the center of $B_{\varepsilon}^{*}$; moreover, this is isomorphic to the algebraic action $\rho_{\alg}^{(p')}$ via an isomorphism of prime-to-$p$ \'{e}tale fundamental groups which takes the image of a loop wrapping around a given $a_{i}(z_{0})$ to the image of a loop wrapping around $\alpha_{i}$.

We now want to relate this to the action of $G_{K, \mathfrak{p}}$ on the prime-to-$p$ \'{e}tale fundamental group of a smooth hyperelliptic curve over $\mathcal{K}_{\mathfrak{p}}$.  Let $\mathfrak{C}$ be a smooth, projective hyperelliptic curve over $\cc$ of degree $d$ and genus $g$ defined by an equation of the form $y^{2} = \prod_{i = 1}^{d} (x - z_{i})$ for distinct roots $z_{i} \in \mathcal{K}_{\mathfrak{p}}$; if $d$ is odd (resp. even), then $d = 2g + 1$ (resp. $d = 2g + 2$).  The hyperelliptic curve $\mathfrak{C}$ comes with a surjective degree-$2$ morphism $\mathfrak{C} \to \proj_{\cc}^{1}$ defined by projecting onto the $x$-coordinate.  It is well known that this projection ramifies at $\infty$ if and only if $d = 2g + 1$; in this case, we write $z_{2g + 2} = \infty$.  Then the projection is ramified at exactly the $(2g + 2)$-element set of $z_{i}$'s.  We denote this subset by $\mathfrak{B} \subset \proj_{\cc}^{1}$, and by abuse of notation we also write $\mathfrak{B} \subset \mathfrak{C}(\cc)$ for the set of inverse images of these ramification points ($\mathfrak{B}$ is the set of \textit{branch points} of $\mathfrak{C}$).  Clearly the restriction to $\mathfrak{C} \smallsetminus \mathfrak{B}$ of the above projection map yields a finite degree-$2$ \'{e}tale morphism $\mathfrak{C} \smallsetminus \mathfrak{B} \to \proj_{\cc}^{1} \smallsetminus \mathfrak{B}$.  Choose a basepoint $P$ of the topological space $\proj_{\cc}^{1} \smallsetminus \mathfrak{B}$ and a basepoint $Q$ of the topological space $\mathfrak{C}(\cc) \smallsetminus \mathfrak{B}$ such that $Q$ lies in the inverse image of $P$.  Then after making identifications via Riemann's Existence Theorem, we get an inclusion and surjections 
\begin{equation} \label{eq maps of fundamental groups} \widehat{\pi}_{1}(\proj_{\cc}^{1} \smallsetminus \mathfrak{B}, P) \rhd \widehat{\pi}_{1}(\mathfrak{C}(\cc) \smallsetminus \mathfrak{B}, Q) \twoheadrightarrow \widehat{\pi}_{1}(\mathfrak{C}(\cc), Q) \twoheadrightarrow H_{1}(\mathfrak{C}(\cc), \zz) \otimes \widehat{\zz} \end{equation}
 induced by the maps $\proj_{\cc}^{1} \smallsetminus \mathfrak{B} \leftarrow \mathfrak{C}(\cc) \smallsetminus \mathfrak{B} \hookrightarrow \mathfrak{C}(\cc)$ and by identifying the first singular homology group of $\mathfrak{C}(\cc)$ with the abelianization of its fundamental group.  The inclusion in (\ref{eq maps of fundamental groups}) is an inclusion of a characteristic subgroup, so any automorphism of $\widehat{\pi}_{1}(\proj_{\cc}^{1} \smallsetminus \mathfrak{B}, P)$ induces an automorphism of $\widehat{\pi}_{1}(\mathfrak{C}(\cc) \smallsetminus \mathfrak{B}, P)$.

We write $\mathfrak{J}$ for the Jacobian of $\mathfrak{C}$.  There is a well-known identification of $H_{1}(\mathfrak{C}(\cc), \zz) \otimes \zz_{\ell}$ with $H_{1}(\mathfrak{J}(\cc), \zz) \otimes \zz_{\ell}$ and in turn with $T_{\ell}(\mathfrak{J})$ for any prime $\ell$.  Moreover, the intersection pairing on $H_{1}(\mathfrak{C}(\cc), \zz)$ defined above carries over to the canonical Riemann form on the complex abelian variety $\mathfrak{J}(\cc)$ and in turn to the Weil pairing $e_{\ell}$ on $T_{\ell}(\mathfrak{J})$ (see the results in \cite[\S IV.4 and \S VIII.1]{lang2012introduction} and in \cite[\S24]{mumford1974abelian}).  Given an element $c \in H_{1}(\mathfrak{C}(\cc), \zz)$, we also write $c$ for the element $c \otimes 1 \in H_{1}(\mathfrak{C}(\cc), \zz) \otimes \zz_{\ell} = T_{\ell}(\mathfrak{J})$.  It is not difficult to show that the action of $G_{K, \mathfrak{p}}$ on $T_{\ell}(\mathfrak{J})$ induced by $\rho_{\alg}^{(p')}$ via these identifications is the natural $\ell$-adic Galois action $\rho_{\ell}$; see, for instance, step 5 of the proof of \cite[Proposition 2.2]{yelton2015images}.

\begin{dfn} \label{dfn symplectic basis}

Given any complex hyperelliptic curve $\mathfrak{C}$ as above, we define the following objects.

a) For any $c \in H_{1}(\mathfrak{C}(\cc), \zz)$, the \textit{transvection} with respect to $c$, denoted $t_{c} \in \Aut(H_{1}(\mathfrak{C}(\cc), \zz))$, is the automorphism given by $v \mapsto v + \langle v, c \rangle c$.  As above, we may identify $c$ with its image in $T_{\ell}(\mathfrak{J})$, and then $t_{c}$ is identified with the automorphism of $T_{\ell}(\mathfrak{J})$ given by $v \mapsto v + e_{\ell}(v, c) c$.

b) A \textit{sympectic basis} of $H_{1}(\mathfrak{C}(\cc), \zz)$ is an ordered basis $\{a_{1}', ... , a_{g}', b_{1}', ... , b_{g}'\}$ satisfying the following properties:

\ \ \ (i) each $a_{i}'$ (resp. each $b_{i}'$) is represented by a loop on $\mathfrak{C}(\cc)$ whose image in $\proj_{\cc}^{1}$ separates $\{z_{2i - 1}, z_{2i}\}$ (resp. $\{z_{2i}, ... , z_{2g + 1}\}$) from its complement in $\{z_{j}\}_{j = 1}^{2g + 2}$; and 

\ \ \ (ii) the (skew-symmetric) intersection pairing $\langle \cdot, \cdot \rangle$ on $H_{1}(\mathfrak{C}(\cc), \zz)$ is determined by $\langle a_{i}', b_{i}' \rangle = -1$ for $1 \leq i \leq g$ and $\langle a_{i}', a_{j}' \rangle = \langle b_{i}', b_{j}' \rangle = \langle a_{i}', b_{j}' \rangle = 0$ for $1 \leq i < j \leq g$.  (See \cite[Figure 6.1]{farb2011primer} or Figure \ref{homology1} below.)

\end{dfn}

\begin{figure}
\caption{Image in $\proj_{\cc}^{1}$ of a symplectic basis of $H_{1}(C(\cc), \zz)$ in the $g = 2$ case} \label{homology1}
\includegraphics[scale=0.25]{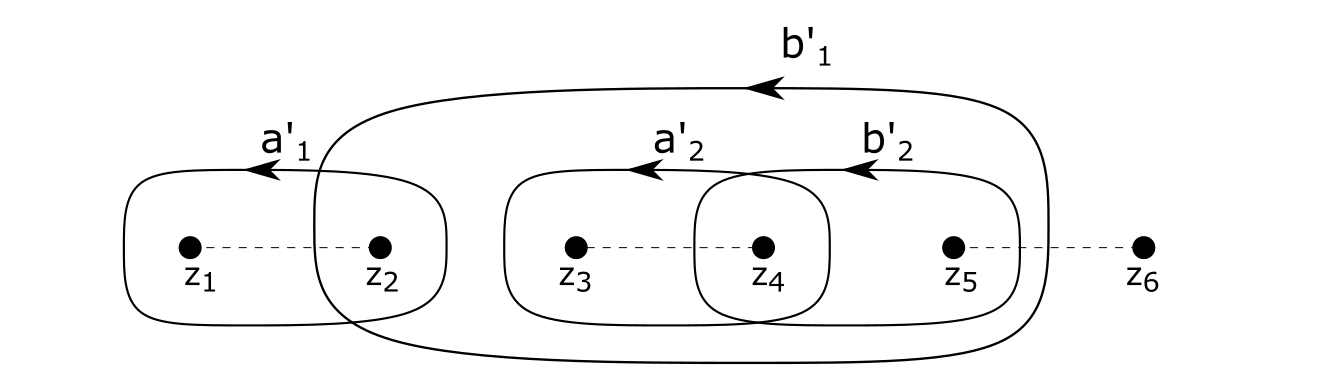}
\end{figure}

We note that a transvection in $\Aut(H_{1}(\mathfrak{C}(\cc), \zz))$ respects the intersection pairing $\langle \cdot, \cdot \rangle$; similarly, a transvection in $\Aut(T_{\ell}(\mathfrak{J}))$ respects the Weil pairing $e_{\ell}$ and thus lies in $\Sp(T_{\ell}(\mathfrak{J}))$.

From now on, given a complex hyperelliptic curve $\mathfrak{C}$, we fix a symplectic basis $\{a_{1}', ... , a_{g}', b_{1}', ... , b_{g}'\}$ of $H_{1}(\mathfrak{C}(\cc), \zz)$.  In order to prove Proposition \ref{prop Galois action local} below, we require several lemmas, the first two of which are purely topological.

\begin{lemma} \label{lemma even partitions}

Via the relations given in (\ref{eq maps of fundamental groups}), there is an identification between the homology group $H_{1}(\mathfrak{C}(\cc), \zz / 2\zz)$ and the group of partitions of $\mathfrak{B}$ into $2$ even-cardinality subsets (where the addition law is induced by the symmetric difference operation).

\end{lemma}

\begin{proof}

The essential ideas of this argument are contained in the proof of \cite[Lemma 8.12]{mumford1984tata}, where the assertion is mentioned in a footnote.  The maximal abelian exponent-$2$ quotient of $\pi_{1}(\proj_{\cc}^{1} \smallsetminus \mathfrak{B}, P)$ is isomorphic to $(\zz / 2\zz)^{2g + 1}$ and is identified with the group of partitions of $\mathfrak{B}$ into two subsets (where the addition law is induced by taking symmetric differences), by sending the homology class of any loop $\gamma \subset \proj_{\cc}^{1}$ to the subsets of $\mathfrak{B}$ contained in each connected component of $\proj_{\cc}^{1} \smallsetminus \gamma$.  We have a homomorphism of homology groups (composed with reduction modulo $2$) $H_{1}(\mathfrak{C}(\cc) \smallsetminus \mathfrak{B}, \zz) \to H_{1}(\proj_{\cc}^{1} \smallsetminus \mathfrak{B}, \zz / 2\zz)$ coming from the inclusion of fundamental groups in (\ref{eq maps of fundamental groups}).  By \cite[Lemma 1]{arnold1968remark}, this homomorphism factors through the surjection $H_{1}(\mathfrak{C}(\cc) \smallsetminus \mathfrak{B}, \zz) \to H_{1}(\mathfrak{C}(\cc), \zz/2\zz)$ induced by the map $\mathfrak{C}(\cc) \smallsetminus \mathfrak{B} \hookrightarrow \mathfrak{C}(\cc)$ (composed with reduction modulo $2$).  It is clear from this that we have an inclusion of $H_{1}(\mathfrak{C}(\cc), \zz / 2\zz)$ as the subgroup of $H_{1}(\proj_{\cc}^{1} \smallsetminus \{z_{1}, ... , z_{2g + 2}\}, \zz / 2\zz)$ which is generated by the partitions $\bar{a}_{i}' := \{z_{2i - 1}, z_{2i}\} \sqcup \mathfrak{B} \smallsetminus \{z_{2i - 1}, z_{2i}\}$ (the image of $a_{i}'$) and $\bar{b}_{i}' := \{z_{2i}, ... , z_{2g + 1}\} \sqcup \mathfrak{B} \smallsetminus \{z_{2i}, ... , z_{2g + 1}\}$ (the image of $b_{i}'$) for $1 \leq i \leq g$.  Now we have $\{z_{2i - 1}, z_{2g + 2}\} \sqcup \mathfrak{B} \smallsetminus \{z_{2i - 1}, z_{2g + 2}\} = \sum_{j = 1}^{i - 1} \bar{a}_{j}' + \sum_{j = i}^{g} \bar{b}_{j}'$ for $1 \leq i \leq g$ and $\{z_{2i}, z_{2g + 2}\} \sqcup \mathfrak{B} \smallsetminus \{z_{2i}, z_{2g + 2}\} = \sum_{j = 1}^{i} \bar{a}_{j}' + \sum_{j = i}^{g} \bar{b}_{j}'$ for $1 \leq i \leq g - 1$. The assertion now follows from checking that given any even-cardinality subset $T \subset \mathfrak{B} \smallsetminus \{z_{2g + 2}\}$, the partition $T \sqcup \mathfrak{B} \smallsetminus T$ is equal to $\sum_{z_{i} \in T} (\{z_{i}, z_{2g + 2}\} \sqcup \mathfrak{B} \smallsetminus \{z_{i}, z_{2g + 2}\})$.

\end{proof}

\begin{lemma} \label{lemma lifts to Dehn twists}

Assume all of the above notation.

a) Let $\gamma \subset \proj_{\cc}^{1}$ be the image of a simple closed loop on $\proj_{\cc}^{1} \smallsetminus \mathfrak{B}$ which separates the set of $\alpha_{i}$'s into two even-cardinality subsets.  Then the inverse image of $\gamma$ under the ramified degree-$2$ covering map $\mathfrak{C}(\cc) \to \proj_{\cc}$ consists of two connected components, each of which are simple closed loops whose homology classes $\pm c \in H_{1}(\mathfrak{C}(\cc), \zz)$ differ by sign.

As a particular case, if $\gamma_{i} \subset \proj_{\cc}^{1}$ is the image of a loop which separates $\{z_{i}, z_{2g + 1}\}$ from its complement in $\mathfrak{B}$ for some $i$, then the homology classes of the simple closed loops on $\mathfrak{C}(\cc)$ lying above $\gamma_{i}$ are given by $\pm c_{i}'$ for some element $c_{i}' \in H_{1}(\mathfrak{C}(\cc), \zz)$ which is equivalent modulo $2$ to 
\begin{equation} \label {eq c_i} \begin{cases} a_{(i + 1)/2}' + ... + a_{g}' + b_{(i + 1)/2}' & i \ \mathrm{odd}\\ a_{i/2 + 1}' + ... + a_{g}' + b_{i/2}' & i \ \mathrm{even} \end{cases} \end{equation}

b) With $\gamma$ and $c$ as above, the Dehn twist $D_{[\gamma]}$ induces an automorphism of $H_{1}(\mathfrak{C}(\cc), \zz)$ via the inclusion and quotient maps in (\ref{eq maps of fundamental groups}), which is given by $t_{c}^{2} \in \Aut(H_{1}(\mathfrak{C}(\cc), \zz))$.

\end{lemma}

\begin{proof}

It follows from Lemma \ref{lemma even partitions} and its proof that if $\gamma$ is any loop in $\pi_{1}(\proj_{\cc}^{1} \smallsetminus \mathfrak{B}, P)$ which seperates the $z_{i}$'s into $2$ subsets of even cardinality, then $\gamma$ lifts to an element $c \in \pi_{1}(\mathfrak{C}(\cc) \smallsetminus \mathfrak{B}, Q)$ whose image in $H_{1}(\mathfrak{C}(\cc), \zz)$ we also denote by $c$. The only other choice of loop on $\mathfrak{C}(\cc) \smallsetminus \mathfrak{B}$ whose image in $\proj_{\cc}^{1} \smallsetminus \mathfrak{B}$ is $\gamma$ must be based at $\iota(Q)$ and equal to the composition of the path $c$ with $\iota$, where $\iota : \mathfrak{C}(\cc) \smallsetminus \mathfrak{B} \to \mathfrak{C}(\cc) \smallsetminus \mathfrak{B}$ is the only nontrivial deck transformation.  Since $\iota$ acts on the homology group by sign change (see \cite[\S7.4]{farb2011primer}), this loop must be $-c$.

If the image of $\gamma$ modulo $2$ is the partition of $\{z_{i}, z_{2g + 1}\}$ and its complement, then the desired statement in (a) follows by the straightforward verification that this partition is equal to the symmetric sum of the partitions corresponding to the basis elements $a_{i}'$ and $b_{i}'$ appearing in the formulas given in (\ref{eq c_i}).  Thus, (a) is proved.

Consider a self-homeomorphism of $\mathfrak{C}(\cc) \smallsetminus \mathfrak{B}$ fixing $Q$ representing the composition of (commuting) Dehn twists on $\mathfrak{C}(\cc) \smallsetminus (\mathfrak{B} \cup \{Q\})$ with respect to the two lifts $\pm c$ of $\gamma$ on $\mathfrak{C}(\cc) \smallsetminus (\mathfrak{B} \cup \{Q\})$; since Dehn twists do not depend on the orientation of loops, this composition of Dehn twists is equal to $D_{c}^{2}$.  Note that such a self-homeomorphism of $\mathfrak{C}(\cc) \smallsetminus \mathfrak{B}$ is (up to deformation) a lifting of a self-homeomorphism of $\proj_{\cc}^{1} \smallsetminus \mathfrak{B}$ fixing $P$ which represents $D_{[\gamma]}$.  Now it follows from the functoriality of taking fundamental groups that the induced automorphism of $\pi_{1}(\mathfrak{C}(\cc) \smallsetminus \mathfrak{B}, Q)$ is the restriction of the automorphism $\varphi(D_{[\gamma]}) \in \Aut(\pi_{1}(\proj_{\cc}^{1} \smallsetminus \mathfrak{B}, P))$.  Since any self-homeomorphism of $\mathfrak{C}(\cc) \smallsetminus \mathfrak{B}$ sends the image of a loop wrapping around a single point in $\mathfrak{B}$ to the image of a loop wrapping around a single point in $\mathfrak{B}$ and the kernel of the quotient map $\pi_{1}(\mathfrak{C}(\cc) \smallsetminus \mathfrak{B}, Q) \twoheadrightarrow \pi_{1}(\mathfrak{C}(\cc), Q)$ is generated by squares of homotopy classes of such loops, this restriction of $\varphi(D_{[\gamma]})$ stabilizes the kernel and induces an automorphism of $\pi_{1}(\mathfrak{C}(\cc), Q)$.  Thus, $D_{[\gamma]}$ induces an automorphism of $H_{1}(\mathfrak{C}(\cc), \zz)$ via the maps in (\ref{eq maps of fundamental groups}), as claimed, and this automorphism is the one induced by the element $D_{c}^{2}$ of the mapping class group of $\mathfrak{C}(\cc)$.  Now (b) follows from the well-known fact (see \cite[Proposition 6.3]{farb2011primer}) that the Dehn twist with respect to any loop on a compact smooth manifold acts on homology as the transvection with respect to that loop.

\end{proof}

\begin{lemma} \label{lemma branch point infty}

Let $\alpha_{1}, ... , \alpha_{d}$ be distinct elements in $\mathcal{R}_{\mathfrak{p}}$.  There exist elements $\alpha_{1}', ... , \alpha_{d + 1}' \in \mathcal{R}_{\mathfrak{p}}$ satisfying the following: 

(i) The elements $\alpha_{i}' - \alpha_{j}'$ and $\alpha_{i} - \alpha_{j}$ have equal valuation for $1 \leq i < j \leq d$, and $\alpha_{d + 1}' - \alpha_{i}' \in \mathcal{R}_{\mathfrak{p}}^{\times}$ for $1 \leq i \leq d$.

(ii) Let $\alpha_{d + 1} = \infty \in \proj_{\mathcal{K}_{\mathfrak{p}}}^{1}$.  There is a $\mathcal{K}_{\mathfrak{p}}$-isomorphism 
$$\psi : \proj_{\bar{\mathcal{K}}_{\mathfrak{p}}}^{1} \smallsetminus \{\alpha_{1}, ... , \alpha_{d + 1}\} \stackrel{\sim}{\to} \proj_{\bar{\mathcal{K}}_{\mathfrak{p}}}^{1} \smallsetminus \{\alpha_{1}', ... , \alpha_{d + 1}'\}$$
 such that the induced isomorphism $\psi_{*} : \widehat{\pi}_{1}(\proj_{\cc}^{1} \smallsetminus \{\alpha_{1}, ... , \alpha_{d + 1}\}, \psi^{-1}(\infty)) \stackrel{\sim}{\to} \widehat{\pi}_{1}(\proj_{\cc}^{1} \smallsetminus \{\alpha_{1}', ... , \alpha_{d + 1}'\}, \infty)$ yields an isomorphism of $\rho_{\alg}$ with the analogously defined representation $\rho_{\alg}'$.  The isomorphism $\psi_{*}$ takes any element represented by a loop on $\proj_{\cc}^{1} \smallsetminus \{\alpha_{1}, ... , \alpha_{d + 1}\}$ separating some singleton $\{\alpha_{i}\}$ from its complement in $\{\alpha_{j}\}_{j = 1}^{d + 1}$ to an element represented by a loop on $\proj_{\cc}^{1} \smallsetminus \{\alpha_{1}', ... , \alpha_{d + 1}'\}$ separating the singleton $\{\alpha_{i}'\}$ from its complement in $\{\alpha_{j}'\}_{j = 1}^{d + 1}$.

\end{lemma}

\begin{proof}

Choose $\beta \in \mathcal{R}_{\mathfrak{p}}^{\times}$ satisfying $\beta \not\equiv \alpha_{j}$ (mod $\pi$) for $1 \leq j \leq d$ (this is always possible because the residue field $\mathcal{R}_{\mathfrak{p}} / (\pi)$ is infinite).  Let $\alpha_{j}' = \alpha _{j}\beta / (\beta - \alpha_{j}) \in \mathcal{R}_{\mathfrak{p}}$ for $1 \leq j \leq d$ and $\alpha_{d + 1}' = \beta \in \mathcal{R}_{\mathfrak{p}}$, and let $\psi$ be the $\mathcal{K}_{\mathfrak{p}}$-morphism given by $x \mapsto x\beta / (\beta - x)$.  Then property (i) follows from straightforward computations.  Since $\psi$ is defined over $\mathcal{K}_{\mathfrak{p}}$, the isomorphism $\psi_{*}$ is equivariant with respect to the action of $G_{K, \mathfrak{p}} = \Gal(\bar{\mathcal{K}}_{\mathfrak{p}} / \mathcal{K}_{\mathfrak{p}})$.  Moreover, after base change to $\cc$, $\psi$ is a homeomorphism of punctured Riemann spheres, and the property given in (ii) follows.

\end{proof}

We are finally ready to state and prove the main result of this section, which is essentially a more concrete version of a particular case of Grothendieck's criterion for semistable reduction (\cite[Proposition 3.5(iv)]{grothendieck1972modeles}).  For the statement of the below proposition, we fix a symplectic basis $\{a_{1}, ... , a_{g}, b_{1}, ... , b_{g}\}$ of $H_{1}(C(\cc), \zz)$ (see Figure \ref{homology1} for an example); the image $\{a_{1}, ... , a_{g}, b_{1}, ... , b_{g}\} \subset T_{\ell}(J)$ forms a symplectic basis of $T_{\ell}(J)$ with respect to the Weil pairing.

\begin{prop} \label{prop Galois action local}

Let $C$ be a hyperelliptic curve of genus $g$ over $\mathcal{K}_{\mathfrak{p}}$ given by an equation of the form $y^{2} = h(x)$ for some squarefree polynomial $h$ of degree $d = 2g + 1$ or $d = 2g + 2$ with distinct roots $\alpha_{1}, ... , \alpha_{d} \in \mathcal{R}_{\mathfrak{p}}$, and let $J$ be its Jacobian.  Suppose that exactly $2$ of the roots, $\alpha_{i}$ and $\alpha_{j}$, are equivalent modulo $\pi$, and let $m \geq 1$ be the maximal integer such that $\pi^{m} \mid (\alpha_{i} - \alpha_{j})$.  Then for any prime $\ell \neq p$, the image of the natural action of $G_{K, \mathfrak{p}}$ on $T_{\ell}(J)$ is topologically generated by the element $t_{c}^{2m} \in \Sp(T_{\ell}(J))$ for some $c \in T_{\ell}(J)$ determined by a loop on $\proj_{\cc}^{1}$ whose image separates $\{\alpha_{i}, \alpha_{j}\}$ from the rest of the roots.

As a particular case, if $i \leq 2g$ and $j = 2g + 1$, then this $\ell$-adic Galois image is topologically generated by $t_{c_{i}}^{2m}$ for some $c_{i} \in H_{1}(C(\cc), \zz)$ equivalent modulo $2$ to $a_{(i + 1)/2} + ... + a_{g} + b_{(i + 1)/2}$ (resp. $a_{i/2 + 1} + ... + a_{g} + b_{i/2}$) if $i$ is odd (resp. even).

\end{prop}

\begin{proof}

We first assume that $d = 2g + 2$.  Let $\tilde{\alpha}_{1}, ... , \tilde{\alpha}_{2g + 2} \in \cc[x]$ be the elements constructed from the $\alpha_{i}$'s as above, and define the family $\mathcal{F} \to B_{\varepsilon}^{*}$ as above.  It is clear from the hypothesis on the roots $\alpha_{j}$ that the set $\mathcal{I}$ in the statement of Theorem \ref{thm comparison punctured projective line} consists of only the elements $(\{i, 2g + 1\}, n)$ for $1 \leq n \leq m$.  Theorem \ref{thm comparison punctured projective line}(a) then implies that a topological generator of $\widehat{\pi}_{1}(B_{\varepsilon}^{*}, z_{0})$ acts on $\widehat{\pi}_{1}(\mathcal{F}_{z_{0}}, \infty)$ via the monodromy action $\rho_{\tp}$ as $D_{[\gamma_{i}]}^{m}$, where $\gamma_{i}$ is a loop on $\mathcal{F}_{z_{0}} \smallsetminus \{\infty\}$ whose image in $\proj_{\cc}^{1}$ separates the subset $\{\tilde{\alpha}_{i}(z_{0}), \tilde{\alpha}_{2g + 1}(z_{0})\}$ from its complement in $\{\tilde{\alpha}_{j}(z_{0})\}_{j = 1}^{2g + 2}$.  Let $\mathfrak{C}$ be the complex hyperelliptic curve of degree $d = 2g + 2$ ramified over $\proj_{\cc}^{1}$ at the points $\tilde{\alpha}_{1}(z_{0}), ... , \tilde{\alpha}_{d}(z_{0}) \in \cc$, which has a symplectic basis $\{a_{1}', ... , a_{g}', b_{1}', ... , b_{g}'\}$ of $H_{1}(\mathfrak{C}(\cc), \zz)$.  By Lemma \ref{lemma lifts to Dehn twists}, the automorphism of $\widehat{\pi}_{1}(\mathcal{F}_{z_{0}}, \infty)$ determined by $D_{[\gamma_{i}]}^{m}$ induces the automorphism of $H_{1}(\mathfrak{C}(\cc), \zz) \otimes \widehat{\zz}$ given by $t_{c_{i}}^{2m} : v \mapsto v + 2m \langle v, c_{i}' \rangle c_{i}'$, where $c_{i}' \in H_{1}(\mathfrak{C}(\cc), \zz)$ is equivalent modulo $2$ to the formula in (\ref{eq c_i}).  Now parts (b) and (c) of Theorem \ref{thm comparison punctured projective line} say that there is an isomorphism $\phi : \widehat{\pi}_{1}(\mathcal{F}_{z_{0}}, \infty)^{(p')} \stackrel{\sim}{\to} \widehat{\pi}_{1}(\proj_{\cc}^{1} \smallsetminus \{\alpha_{1}, ... , \alpha_{d}\}, \infty)^{(p')}$ making the action $\rho_{\tp}$ isomorphic to the Galois action $\rho_{\alg}$.  Since $p \neq 2$ and $\mathfrak{C}(\cc) \smallsetminus \mathfrak{B}$ and $C(\cc) \smallsetminus \{(\alpha_{i}, 0)\}_{i = 1}^{2g + 2}$ are the only degree-$2$ covers of $\mathcal{F}_{z_{0}}$ and $\proj_{\cc}^{1} \smallsetminus \{\alpha_{1}, ... , \alpha_{2g + 2}\}$ respectively, we see that the isomorphism $\phi$ induces an isomorphism $H_{1}(\mathfrak{C}(\cc), \zz) \otimes \widehat{\zz}^{(p')} \stackrel{\sim}{\to} H_{1}(C(\cc), \zz) \otimes \widehat{\zz}^{(p')}$ which we also denote by $\phi$.  It is clear from the property of $\phi$ given in Theorem \ref{thm comparison punctured projective line}(c) and from our characterization of homology groups with coefficients in $\zz / 2\zz$ in the proof of Lemma \ref{lemma lifts to Dehn twists} that $\phi(a_{j}') \equiv a_{j}$ and $\phi(b_{j}') \equiv b_{j}$ (mod $2$) for $1 \leq j \leq g$.

In the case that $d = 2g + 1$, we get the same results by applying Lemma \ref{lemma branch point infty}, which allows us to replace $\alpha_{1}, ... , \alpha_{2g + 1}, \alpha_{2g + 2} := \infty$ with elements $\alpha_{1}', ... , \alpha_{2g + 2}' \in \mathcal{R}_{\mathfrak{p}}$ whose differences have the same valuations with respect to $\pi$.

Putting this all together, we see that the tame Galois action on $H_{1}(C(\cc), \zz) \otimes \widehat{\zz}^{(p')}$ induced by $\rho_{\alg}$ sends a generator of $G_{K, \mathfrak{p}}^{\tame}$ to the automorphism of $H_{1}(C(\cc), \zz) \otimes \widehat{\zz}^{(p')}$ given by $v \mapsto v + 2m \langle v, c_{i} \rangle_{\phi} c_{i}$ for some $c_{i} \in H_{1}(C(\cc), \infty) \otimes \widehat{\zz}^{(p')}$ which is equivalent modulo $2$ to the formula given in the statement.  Here $\langle \cdot, \cdot \rangle_{\phi}$ is the skew-symmetric pairing on $H_{1}(C(\zz), \zz) \otimes \widehat{\zz}^{(p')}$ induced by the interection pairing on $H_{1}(\mathfrak{C}, \zz)$ via $\phi$; note that $\langle \cdot, \cdot \rangle_{\phi}$ is normalized so that $\langle H_{1}(C(\cc), \zz), H_{1}(C(\cc), \zz) \rangle_{\phi} = \zz$.  As above, we identify the maximal pro-$\ell$ quotient of $H_{1}(C(\cc), \zz) \otimes \widehat{\zz}^{(p')}$ with $T_{\ell}(J)$ and see that the natural $\ell$-adic Galois action $\rho_{\ell}$ factors through $G_{K, \mathfrak{p}}^{\tame}$ and takes a generator to the automorphism of $T_{\ell}(J)$ given by $v \mapsto v + 2m\langle v, c_{i} \rangle_{\phi} c_{i}$.  But this automorphism must lie in $\Sp(T_{\ell}(J))$ by the Galois equivarience of the Weil pairing $e_{\ell}$ and the fact that $\mathcal{K}_{\mathfrak{p}}$ contains all $\ell$-power roots of unity.  It is now an easy exercise to verify that this implies that $\langle v, c_{i} \rangle_{\phi} = \pm e_{\ell}(v, c_{i})$ for all $v \in T_{\ell}(J)$, and so the image of $\rho_{\ell}$ is generated by $t_{c_{i}}^{2m} : v \mapsto v + 2m e_{\ell}(v, c_{i}) c_{i}$, as desired.

\end{proof}

In order to prove Theorems \ref{thm main1} and \ref{thm main2}, we will put some local data together to show using Proposition \ref{prop Galois action local} that the $\ell$-adic Galois image contains $\{t_{c_{1}}^{2m}, ... , t_{c_{2g}}^{2m}\}$ for homology classes $c_{i}$ as above and a certain integer $m$.  Therefore, it is of interest to investigate the subgroup of $\Sp(T_{\ell}(J))$ that this set generates.  We will see that this subgroup is always open, but unfortunately, since the $c_{i}$'s are only known modulo $2$, not much more can be deduced except if $\ell = 2$.  In this case, the subgroup of $\Sp(T_{2}(J))$ generated by the above set can be determined precisely, which is the goal of the next section.

\section{Subgroups generated by powers of transvections} \label{S3}

For this section, we let $M$ be a free $\zz$-module of rank $2g$, equipped with a nondegenerate skew-symmetric $\zz$-bilinear pairing $\langle \cdot, \cdot \rangle : M \times M \to \zz$.  For any ring $A$, this pairing induces in an obvious way a nondegenerate skew-symmetric $A$-bilinear pairing on the free $A$-module $M \otimes A$ which we also denote by $\langle \cdot, \cdot \rangle$.  We write $\Sp(M \otimes A)$ for the symplectic group of $A$-automorphisms of $M \otimes A$ which respect this pairing; for any integer $N \geq 1$, we write $\Gamma(N)$ for the level-$N$ congruence subgroup of $\Sp(M \otimes A)$ as defined in \S\ref{S1}.  For any element $c \in M \otimes A$, we write $t_{c} \in \Sp(M \otimes A)$ for the transvection with respect to $c$ and the pairing $\langle \cdot, \cdot \rangle$, as in Definition \ref{dfn symplectic basis}(a).  We fix a basis $\{a_{1}, ... , a_{g}, b_{1}, ... , b_{g}\}$ of $M$ and assume that the pairing is determined by $\langle a_{i}, b_{i} \rangle = -1$ for $1 \leq i \leq g$ and $\langle a_{i}, a_{j} \rangle = \langle b_{i}, b_{j} \rangle = \langle a_{i}, b_{j} \rangle = 0$ for $1 \leq i < j \leq g$.  We also write $a_{i}, b_{i} \in M \otimes A$ for the elements $a_{i} \otimes 1$ and $b_{i} \otimes 1$ respectively for $1 \leq i \leq g$.  When $A$ is a field, we write $\mathfrak{sp}(M \otimes A)$ for the Lie algebra of endomorphisms of $M \otimes A$ corresponding to the Lie group $\Sp(M \otimes A)$ which is equipped with the Lie bracket $[\cdot, \cdot]$; it is well known that $\mathfrak{sp}(M \otimes A)$ is a $(2g^2 + g)$-dimensional vector space over $A$.  Our main goal is to prove the following purely algebraic result.

\begin{prop} \label{prop open subgroup}

Let $c_{1}, ... , c_{2g} \in M$ be elements such that $c_{i}$ is equivalent modulo $2$ to 
\begin{equation} \begin{cases} a_{(i + 1)/2} + ... + a_{g} + b_{(i + 1)/2} & i \ \mathrm{odd}\\ a_{i/2 + 1} + ... + a_{g} + b_{i/2} & i \ \mathrm{even} \end{cases} \end{equation}
 for $1 \leq i \leq 2g$, and let $\ell$ be any prime.  Let $G'_{\ell} \subseteq \Sp(M \otimes \zz_{\ell})$ be the subgroup topologically generated by the elements $t_{c_{1}}^{m_{1}}, ... , t_{c_{2g}}^{m_{2g}}$ for some integers $m_{i} \in \zz \smallsetminus \{0\}$.  Then $G'_{\ell}$ is open as a Lie subgroup of $\Sp(M \otimes \zz_{\ell})$.

In the case that $\ell = 2$, assume moreover that for some integers $n, n' \geq 1$, we have $v_{2}(m_{i}) \leq n$ for $1 \leq i \leq 2g - 1$ and $v_{2}(m_{2g}) = n'$.  Then $G_{2}'$ contains $\Gamma(2^{2n})$ (resp. $\Gamma(2^{n + n'})$) if $n' \leq n$ (resp. if $n' > n$).

\end{prop}

Note that for any prime $\ell$, the image of any transvection $t \in \Sp(M \otimes \zz_{\ell})$ under the logarithm map is $t - 1 \in \mathfrak{sp}(M \otimes \qq_{\ell})$ and that the Lie algebra of $G'_{\ell}$ is generated as a Lie algebra by the logarithms of the transvections $t_{c_{i}}$, which are given by $T_{i} := t_{c_{i}} - 1$.  Therefore, in order to prove the first statement of this proposition, it suffices to show that $\{T_{i}\}_{1 \leq i \leq 2g}$ generates the full $\ell$-adic symplectic Lie algebra $\mathfrak{sp}(M \otimes \qq_{\ell})$.  However, in order to get both statements, we will prove something slightly stronger which is implied by the following lemma.

\begin{lemma} \label{lemma transvection commutator basis}

Let $\bar{T}_{i} \in \mathfrak{sp}(M \otimes \ff_{2})$ denote the image modulo $2$ of $t_{c_{i}} - 1 \in \mathfrak{sp}(M \otimes \zz_{2})$.  Then an $\ff_{2}$-basis for $\mathfrak{sp}(M \otimes \ff_{2})$ is given by the set $\{\bar{T}_{i}\}_{1 \leq i \leq 2g} \cup \{[\bar{T}_{i}, \bar{T}_{j}]\}_{1 \leq i < j \leq 2g}$.

\end{lemma}

\begin{proof}

Since the vector space $\mathfrak{sp}(M \otimes \ff_2)$ has dimension $2g^2 + g$ over $\ff_2$ and there are $2g^{2} + g$ elements listed in the set of $\bar{T}_{i}$'s and their commutators, it suffices to prove that this set is linearly independent over $\ff_{2}$.

Assume that for some scalars $\gamma_{i, j} = \gamma_{j, i} \in \ff_{2}$ for $1 \leq i \leq j \leq 2g$, we have 
\begin{equation} \sum_{i = 1}^{2g} \gamma_{i, i} \bar{T}_{i} + \sum_{1 \leq i < j \leq 2g} \gamma_{i, j} [\bar{T}_{i}, \bar{T}_{j}] = 0. \end{equation}
Let $\bar{c}_i$ denote the reduction modulo $2$ of $c_i \in M$ for $1 \leq i \leq 2g$; from the characterization of the $c_i$'s modulo $2$ we see that the set of elements $\bar{c}_i$ is linearly independent over $M \otimes \ff_2$.  Note that each commutator $[\bar{T}_{i}, \bar{T}_{j}]$ is the map given by $v \mapsto \langle v, \bar{c}_{j} \rangle \bar{c}_{i} + \langle v, \bar{c}_{i} \rangle \bar{c}_{j}$.  Thus, for all $v \in M \otimes \ff_{2}$, we have 
$$0 = \sum_{i = 1}^{2g} \gamma_{i, i} \langle v, \bar{c}_{i} \rangle \bar{c}_{i} + \sum_{1 \leq i < j \leq 2g} \gamma_{i, j} (\langle v, \bar{c}_{i} \rangle \bar{c}_{j} + \langle v, \bar{c}_{j} \rangle \bar{c}_{i}) = \sum_{i = 1}^{2g} \left(\sum_{j = 1}^{2g} \gamma_{i, j} \langle v, \bar{c}_{j} \rangle \bar{c}_{i}\right) = \sum_{i = 1}^{2g} \langle v, \textstyle\sum_{j = 1}^{2g} \gamma_{i, j} \bar{c}_{j} \rangle \bar{c}_{i}.$$
Since the $\bar{c}_{i}$'s are linearly independent, it follows that $\langle v, \sum_{j = 1}^{2g} \gamma_{i, j} \bar{c}_{j} \rangle \equiv 0$ for each $i$.  Since our symplectic pairing is nondegenerate, we have $\sum_{j = 1}^{2g} \gamma_{i, j} \bar{c}_{j} = 0$ for each $i$, and therefore all of the $\gamma_{i, j}$'s are equal to $0$.  The claim follows.

\end{proof}

\begin{proof}[Proof (of Proposition \ref{prop open subgroup})]

Consider the set of $T_{i}$'s and their commutators viewed as endomorphisms of the free $\zz$-module $M$.  Lemma \ref{lemma transvection commutator basis} says that their images modulo $2$ are linearly independent over $\ff_{2}$.  Applying Nakayama's Lemma, we get that the set of $T_{i}$'s and their commutators is $\zz$-linearly independent, and therefore, for each prime $\ell$, the set $\{T_{i}\}_{1 \leq i \leq 2g} \cup \{[T_{i}, T_{j}]\}_{1 \leq i < j \leq 2g} \subset \mathfrak{sp}(M \otimes \qq_{\ell})$ is $\qq_{\ell}$-linearly independent.  Since this set consists of $2g^{2} + g$ elements and the symplectic Lie algebra $\mathfrak{sp}(M \otimes \qq_{\ell})$ is $(2g^{2} + g)$-dimensional, the set generates all of $\mathfrak{sp}(M \otimes \qq_{\ell})$.  Therefore, the Lie algebra of $G'_{\ell}$, which clearly contains the $T_{i}$'s and their commutators, coincides with $\mathfrak{sp}(M \otimes \qq_{\ell})$.  The first statement of the proposition immediately follows.

Now we assume that $\ell = 2$ and prove the second statement.  Note that it is meaningful to raise transvections to the powers of elements in $\zz_{2}$; therefore, since $2^{n'} / m_{2g}, 2^{n} / m_{i} \in \zz_{2}$ for $1 \leq i \leq 2g - 1$, it suffices to prove the statement under the assumption that $m_{1} = ... = m_{2g - 1} = 2^{n}$ and $m_{2g} = 2^{n'}$.

For ease of notation, we write $n'' = \max \{n, n'\}$ and $N = n + n''$, so that the desired statement is that $G'_{2}$ contains the congruence subgroup $\Gamma(2^{N})$.  Since $G'_{2}$ is closed, it suffices to show that the image of $G$ modulo $2^{N + m}$ contains $\Gamma(2^{N}) / \Gamma(2^{N + m})$ for each integer $m \geq 1$.  We claim that in fact we only need to show this for $m = 1$.  Indeed, for any $m \geq 1$, the restriction of the logarithm map to $\Gamma(2)$ sends each element $t \in \Gamma(2^{N + m - 1})$ to an element of $\mathfrak{sp}(M \otimes \qq_{2})$ which is equivalent to $t - 1$ modulo $2^{N + m}$ since $(t - 1)^{2} \equiv 0$ modulo $2^{N + m}$.  In this way, one verifies that there is an isomorphism from the additive group $\mathfrak{sp}(M \otimes \ff_{2})$ to $\Gamma(2^{N + m - 1}) / \Gamma(2^{N + m})$, given by sending an element $T \in \mathfrak{sp}(M \otimes \ff_{2})$ to $1 + 2^{N + m - 1}\tilde{T} \in \Sp(M \otimes \zz_{2})$, where $\tilde{T}$ is any operator in $\End(M \otimes \zz / 2^{N + m}\zz)$ whose image modulo $2$ is $T$.  In particular, each $\Gamma(2^{N + m - 1}) / \Gamma(2^{N + m})$ is an elementary abelian group of exponent $2$ and rank $2g^{2} + g$ (this is also known from the proof of \cite[Corollary 2.2]{sato2010abelianization}).  Moreover, it is easy to check that for each $m \geq 1$, the map sending each matrix in $\Gamma(2^{N})$ to its $2^{m - 1}$th power induces a group isomorphism $\Gamma(2^{N}) / \Gamma(2^{N + 1}) \stackrel{\sim}{\to} \Gamma(2^{N + m - 1}) / \Gamma(2^{N + m})$.  It follows that if the image of $G'_{2}$ modulo $2^{N + 1}$ contains $\Gamma(2^{N}) / \Gamma(2^{N + 1})$, then the image of $G'_{2}$ modulo $2^{N + m}$ contains $\Gamma(2^{N}) / \Gamma(2^{N + m})$ for all $m \geq 1$.  Thus, in order to prove the proposition, it suffices to show that the image of $G'_{2}$ modulo $2^{N + 1}$ contains $\Gamma(2^{N}) / \Gamma(2^{N + 1})$.

As above, let $T_{i}$ denote $t_{c_{i}} - 1$ for $1 \leq i \leq 2g$.  Since Lemma \ref{lemma transvection commutator basis} says that the image modulo $2$ of $\{T_{i}\}_{1 \leq i \leq 2g} \cup \{[T_{i}, T_{j}]\}_{1 \leq i < j \leq 2g}$ is a basis of $\mathfrak{sp}(M \otimes \ff_{2})$, it suffices to show that the image of $G'_{2}$ modulo $2^{N + 1}$ contains the images of the elements in $\{1 + 2^{N}T_{i}\}_{1 \leq i \leq 2g} \cup \{1 + 2^{N}[T_{i}, T_{j}]\}_{1 \leq i < j \leq 2g} \subset \Gamma(2^{2n})$.  Clearly $G'_{2}$ is generated by $\{1 + 2^{n}T_{i}\}_{1 \leq i \leq 2g - 1} \cup \{1 + 2^{n'}T_{2g}\}$.  We verify using the property $T_{i}^{2} = 0$ that $(1 + 2^{n}T_{i})^{2^{n''}} \equiv 1 + 2^{N}T_{i}$ (mod $2^{N + 1}$) for $1 \leq i \leq 2g - 1$, that $(1 + 2^{n'}T_{2g})^{2^{n + n'' - n'}} \equiv 1 + 2^{N}T_{2g}$ (mod $2^{N + 1}$), and that 
$$(1 + 2^{n}T_{i})(1 + 2^{n''}T_{j})(1 + 2^{n}T_{i})^{-1}(1 + 2^{n''}T_{j})^{-1} \equiv 1 + 2^{N}[T_{i}, T_{j}] \ (\mathrm{mod} \ 2^{N + 1})$$
 for $1 \leq i < j \leq 2g$.  Thus, $G'_{2} \supset \Gamma(2^{N})$, as desired.

\end{proof}

We now state and prove another proposition which will be needed only for the last statements of Theorems \ref{thm main1} and \ref{thm main2}, which pertain to the case that the hyperelliptic curve has degree $4$.

\begin{prop} \label{prop open subgroup degree 4}

Assume the notation and assumptions of Proposition \ref{prop open subgroup} and that $g = 1$ and $n' = n$, and let $c_{3} \in M$ be an element which is equivalent modulo $2$ to $a_{1}$.  Then the subgroup of $\Gamma(2^{n})$ generated by $G_{2}$ and the element $t_{c_{3}}^{2^{n}}$ coincides with $\Gamma(2^{n})$.

\end{prop}

\begin{proof}

By the arguments used at the beginning of the proof of Proposition \ref{prop open subgroup}, we only need to show that the images of $t_{c_{1}}^{2^{n}}, t_{c_{2}}^{2^{n}}, t_{c_{3}}^{2^{n}}$ modulo $2^{n + 1}$ generate $\Gamma(2^{n}) / \Gamma(2^{n + 1})$.  Using the isomorphism from the additive group $\mathfrak{sp}(M \otimes \ff_{2})$ to $\Gamma(2^{n}) / \Gamma(2^{n + 1})$ which was established in that proof, we see that it suffices to show that $\{\bar{T}_{1}, \bar{T}_{2}, \bar{T}_{3}\}$ is linearly independent and thus generates the rank-$3$ elementary abelian group $\mathfrak{sp}(M \otimes \ff_{2})$, where $\bar{T}_{i}$ denotes the image modulo $2$ of $t_{c_{i}} - 1$ for $1 \leq i \leq 3$.  But this is clear from noting that $c_{3} \equiv c_{1} + c_{2}$ (mod $2$) and writing out these linear operators as matrices with respect to an ordered basis of $M \otimes \ff_{2}$ consisting of the images of $c_{1}$ and $c_{2}$.

\end{proof}

\begin{rmk} \label{rmk graph}

It is possible to prove a more general version of Proposition \ref{prop open subgroup} whose statement is as follows.  Let $c_1, ... , c_{2g} \in M$ be elements which form a basis of $M$, and choose a prime $\ell$ and integers $m_1, ... , m_{2g}$ and define $G_{\ell}' \subseteq \Sp(M \otimes \zz_{\ell})$ as in the statement of the proposition.  Consider the simple graph $\Gamma$ whose vertices are given by the basis elements $c_i$ and whose edge set is defined so that two vertices $c_i$ and $c_j$ are connected by an edge if and only if $\langle c_i, c_j \rangle \equiv 1$ (mod $2$).  Then if the graph $\Gamma$ is connected, we have that $G_{\ell}'$ is open as a subgroup of $\Sp(M \otimes \zz_{\ell})$.  If we assume moreover that $\ell = 2$, given an integer $n \geq 1$ such that $v_2(m_i) \leq n$ for $1 \leq i \leq 2g$, we have $G_2' \supsetneq \Gamma(2^{(\delta + 1)n})$, where $\delta$ is the diameter of the graph $\Gamma$.  These claims can be proved using a similar strategy as in the proof of Proposition \ref{prop open subgroup}, which treats the special case where $\Gamma$ is a complete graph and thus is connected with diameter $\delta = 1$.

\end{rmk}

\section{Proof of main theorems and a further result} \label{S4}

The main goal of this section is to prove Theorems \ref{thm main1} and \ref{thm main2}.  Our strategy for this is to put together local results with respect to several primes of $K$ using Proposition \ref{prop Galois action local} to get several elements in $G_{\ell}$ and then to use Proposition \ref{prop open subgroup} to determine that the subgroup of $G_{\ell}$ generated by these elements is open and, when $\ell = 2$, contains a certain congruence subgroup of the symplectic group.  This is realized by the following theorem, which can be used in a far more general situation than is required for Theorems \ref{thm main1} and \ref{thm main2} and is therefore a useful result in its own right.

\begin{thm} \label{thm several primes}

Let $J$ be the Jacobian of the hyperelliptic curve $C$ over $K$ of genus $g$ and degree $d'$ with defining equation $y^{2} = \prod_{i = 1}^{d'} (x - \alpha_{i})$ for some elements $\alpha_{i} \in \mathcal{O}_{K}$ for $1 \leq i \leq d'$, and define the $2$-adic Galois image $G_{2}$ as above.  Suppose that there are distinct primes $\mathfrak{p}_{1}$, ... , $\mathfrak{p}_{d' - 1}$ of $K$ not lying over $(2)$ and such that for $1 \leq i \leq d' - 1$, the only two $\alpha_{j}$'s which are equivalent modulo $\mathfrak{p}_{i}$ are $\alpha_{i}$ and $\alpha_{d'}$; let $m_{i} \geq 1$ be the maximal integer such that $\mathfrak{p}_{i}^{m_{i}} \mid (\alpha_{d'} - \alpha_{i})$.  Let $n = \max \{v_{2}(m_{i})\}_{i = 1}^{d' - 2}$.

a) For any prime $\ell$ different from all of the residue characteristics of the $\mathfrak{p}_{i}$'s, the Lie subgroup $G_{\ell} \subseteq \GSp(T_{\ell}(J))$ is open.  In particular, $G_{2} \subseteq \GSp(T_{2}(J))$ is open.

b) If $n' := v_{2}(m_{d' - 1}) \leq n$ or if $d' = 2g + 2$, we have $\Gamma(2) \supseteq G_{2} \cap \Sp(T_{2}(J)) \supsetneq \Gamma(2^{2n + 2})$.  Otherwise, we have $\Gamma(2) \supseteq G_{2} \cap \Sp(T_{2}(J)) \supsetneq \Gamma(2^{n + n' + 2})$.  Moreover, if $d' = 4$, then $G_{2} \cap \Sp(T_{2}(J)) \supseteq \Gamma(2^{\max \{n, n'\} + 1})$.

\end{thm}

\begin{proof}

Using the embedding $\bar{K} \hookrightarrow \cc$ fixed at the beginning of \S\ref{S2}, we identify $T_{\ell}(J)$ with $H_{1}(C(\cc), \zz) \otimes \zz_{\ell}$ as we did in \S\ref{S2}.  Let $\{a_{1}, ... , a_{g}, b_{1}, ... , b_{g}\}$ be a symplectic basis of $H_{1}(C(\cc), \zz)$ as in Definition \ref{dfn symplectic basis}(b) with $(z_{1}, ... , z_{2g + 2}) = (\alpha_{1}, ... , \alpha_{2g + 1}, \infty)$ if $d'$ is odd and $(z_{1}, ... , z_{2g + 2}) = (\alpha_{1}, ... , \alpha_{2g}, \alpha_{2g + 2}, \alpha_{2g + 1})$ if $d'$ is even.  Then Proposition \ref{prop Galois action local} says that for $1 \leq i \leq 2g$, the image of $G_{K, \mathfrak{p}_{i}} \subset G_{K}$ in $\GSp(T_{\ell}(J))$ contains $t_{c_{i}}^{2m_{i}}$, where $t_{c_{i}} \in \Sp(T_{\ell}(J))$ is the transvection with respect to an element $c_{i} \in T_{2}(J)$ which is equivalent modulo $2$ to the formula given in (\ref{eq c_i}).  Meanwhile, if $d'$ is even, the image of $G_{K, \mathfrak{p}_{2g + 1}}$ contains $t_{c_{d'}}^{2m_{2g + 1}}$ for some other element $c_{d'} \in T_{\ell}(J)$ corresponding to the lift of a loop separating $\{\alpha_{2g + 1}, \alpha_{2g + 2}\}$ from its complement in $\{\alpha_{j}\}_{j = 1}^{2g + 2}$.  In any case, we have $t_{c_{1}}^{2m_{1}}, ... , t_{c_{2g}}^{2m_{2g}} \in G_{\ell}$ for all $\ell$.  It then follows from Proposition \ref{prop open subgroup} applied to $M := H_{1}(C(\cc), \zz)$ that $G_{\ell}$ is open for all $\ell$, and moreover that $G_{2} \supset \Gamma(2^{2n + 2})$ if $n' \leq n$ or if $d' = 2g + 2$ and $G_{2} \supset \Gamma(2^{n + n' + 2})$ otherwise.

If $d' = 4$, then in particular, we have $t_{c_{1}}^{2^{\max\{n, n'\} + 1}}, t_{c_{2}}^{2^{\max\{n, n'\} + 1}}, t_{c_{3}}^{2^{\max\{n, n'\} + 1}} \in G_{2}$ with $c_{1}, c_{2}, c_{3}$ as described above.  Clearly we have $c_{1} \equiv a_{1} + b_{1}$, $c_{2} \equiv b_{1}$, and $c_{3} \equiv a_{1}$ (mod $2$), and so Proposition \ref{prop open subgroup degree 4} implies that $G_{2} \supseteq \Gamma(2^{\max \{n, n'\} + 1})$.

\end{proof}

\begin{rmk} \label{rmk index bound}

We note that it is generally not difficult to compute the order of $G'_{2} / \Gamma(2^{2n + 2})$ or $G'_{2} / \Gamma(2^{n + n' + 2})$, where $G'_{2} \subseteq G_{2}$ is the subgroup generated by the powers of transvections given in the proof above.  Therefore, one may improve the upper bound for $[\Gamma(2) : G_{2}]$ which directly follows from the statement of Theorem \ref{thm several primes}.  For example, we have $[\Gamma(2) : G_{2}] \leq 2^{(2n + 1)(2g^{2} + g) - (n + 1)(d' - 1)}$ in the case that $n = n'$.

\end{rmk}

\begin{ex}[Legendre curve] \label{ex Legendre}

For any $\lambda \in \mathcal{O}_{K} \smallsetminus \{0, 1\}$, let $E_{\lambda}$ be the elliptic curve over $K$ given by $y^{2} = x(x - 1)(x - \lambda)$.  Suppose that there exist (necessarily distinct) primes $\mathfrak{p}_{1}$ and $\mathfrak{p}_{2}$ of $K$ not lying over $(2)$ and integers $m_{1}, m_{2}, \geq 1$ such that $\mathfrak{p}_{1}^{m_{1}}$ exactly divides $(\lambda)$ and $\mathfrak{p}_{2}^{m_{2}}$ exactly divides $(\lambda - 1)$.  Then Theorem \ref{thm several primes} tells us that the $2$-adic Galois image $G_{2}$ (strictly) contains $\Gamma(2^{v_{2}(m_{1}) + v_{2}(m_{2}) + 2})$.  In the case that $m_{1} = m_{2} = 1$ (e.g. $K = \qq$, $\lambda = 6$, $\mathfrak{p}_{1} = (3)$, $\mathfrak{p}_{2} = (5)$), we get $\Gamma(2) \supset G_{2} \cap \Sp(T_{2}(E_{\lambda})) \supsetneq \Gamma(4)$ and can therefore directly compute the precise subgroup $G_{2} \cap \Sp(T_{2}(E_{\lambda})) \subset \Gamma(2)$ using the well-known fact that the $4$-division field $K(E_{\lambda}[4])$ is generated over $K$ by $\{\sqrt{-1}, \sqrt{\lambda}, \sqrt{\lambda - 1}\}$.

It is also possible to prove the statement of Proposition \ref{prop Galois action local}, thus determining $G_{2}$, in the particular cases of $C = E_{\lambda}$ over $\mathcal{K}_{\mathfrak{p}_{1}}$ and over $\mathcal{K}_{\mathfrak{p}_{2}}$ using formulas for generators of $2$-power division fields of $E_{\lambda}$ found in \cite{yelton2015dyadic}, as the author has done in \cite[\S3.4]{yelton2015hyperelliptic}.

\end{ex}

We now prove Theorems \ref{thm main1} and \ref{thm main2} together.

\begin{proof}[Proof (of Theorems \ref{thm main1} and \ref{thm main2})]

First assume the notation and hypotheses of Theorem \ref{thm main1}.  Let $\alpha_{1}, ... , \alpha_{d}$ denote the roots of $f$, and write $L = K(\alpha_{1}, ... , \alpha_{d})$ for the splitting field of $f$ over $K$.  Note that $\Gal(L / K)$ acts transitively on the $\alpha_{i}$'s since $f$ is irreducible.  It then follows from the well-known description of the $2$-division field of a hyperelliptic Jacobian (see for instance \cite[Corollary 2.11]{mumford1984tata}) that $G_{K}$ does not fix the $2$-torsion points of $J$ and so $G_{2}$ is not contained in $\Gamma(2)$, while the image of $\Gal(\bar{K} / L)$ under $\rho_{2}$ coincides with $G_{2} \cap \Gamma(2)$.

The fact that $\mathfrak{p} \nmid (2\Delta)$ implies that the extension $L / K$ is not ramified at $\mathfrak{p}$, and so $\mathfrak{p}$ splits into a product $\mathfrak{p}_{1} ... \mathfrak{p}_{r}$ of distinct primes in $L$, for some integer $r$ dividing $[L : K]$.  Then since $\mathfrak{p}^{m}$ exactly divides $(f(\lambda)) = \prod_{i = 1}^{d}(\lambda - \alpha_{i})$, we have $\mathfrak{p}_{i} \mid (\lambda - \alpha_{i})$ for some $i$; we assume without loss of generality that $\mathfrak{p}_{1} \mid (\lambda - \alpha_{1})$.  Then $\mathfrak{p}_{1}$ cannot divide $(\lambda - \alpha_{i})$ for any $i \in \{2, ... , d\}$, because otherwise for such an $i$ we would have $\mathfrak{p}_{1} \mid (\alpha_{i} - \alpha_{1}) \mid (\Delta)$, which contradicts the hypothesis that $\mathfrak{p} \nmid (2\Delta)$.  It follows that $\mathfrak{p}_{1}^{m}$ exactly divides $(\lambda - \alpha_{1})$.  Then by applying elements of $\Gal(L / K)$ that take $\alpha_{1}$ to each $\alpha_{i}$, we get other primes lying over $\mathfrak{p}$ whose $m$th powers exactly divide the ideals $(\lambda - \alpha_{i})$; we assume without loss of generality that $\mathfrak{p}_{i}^{m}$ exactly divides $(\lambda - \alpha_{i})$ for $1 \leq i \leq d$.  Now since the $\mathfrak{p} \nmid (2\Delta)$ hypothesis implies that none of the $\mathfrak{p}_{i}$'s lie over $(2)$, we can apply Theorem \ref{thm several primes} with $K$ replaced by $L$, $d' = d + 1$, $\alpha_{d'} = \lambda$, and $m_{1} = ... = m_{d' - 1} = m$ to get the statements of Theorem \ref{thm main1}.

Now assume the notation and hypotheses of Theorem \ref{thm main2}.  Then the argument for proving Theorem \ref{thm main2} is the same except that when applying Theorem \ref{thm several primes} we choose $\mathfrak{p}_{d' - 1}$ to be a prime of $L$ lying over $\mathfrak{p}'$, and we put $d' = d + 2$, $\alpha_{d'} = \lambda$, $\alpha_{d' - 1} = \lambda'$, $m = m_{1} = ... = m_{d' - 2}$, and $m' = m_{d' - 1}$.

\end{proof}

\section{Realizing uniform boundedness along one-parameter families} \label{S5}

Fix an irreducible monic polynomial $f \in \mathcal{O}_{K}[x]$ of degree $d \geq 2$.  It follows from Cadoret and Tamagawa's results \cite[Theorems 1.1 and 5.1]{cadoret2012uniform} that for all but finitely many $\lambda \in K$, the $\ell$-adic Galois image $G_{\ell, \lambda}$ associated to the Jacobian of the curve given by $y^{2} = f(x)(x - \lambda)$ is open in the $\ell$-adic Galois image $G_{\ell, \eta}$ associated to the generic fiber of the family parametrized by $\lambda$.  These theorems also assert that there is some integer $B \geq 1$ depending only on $f$ and $\ell$ such that the index of $G_{\ell, \lambda}$ in $G_{\ell, \eta}$ is bounded by $B$ for all but finitely many $\lambda \in K$.  The following theorem treats the $\ell = 2$ case: we recover the openness result when $d \geq 4$ and explicitly provide the aforementioned uniform bound when $d$ is even.  (Note that for the elliptic curve case, where $d \in \{ 2, 3\}$, such openness results are already known from the celebrated Open Image Theorem of Serre given by \cite[IV-11]{serre1989abelian}, while uniform bounds are given by \cite[Theorem 1.3]{arai2008uniform}.)  It is interesting to note that Faltings' Theorem is used both in the proof of \cite[Theorem 1.1]{cadoret2012uniform} and in our proof of the theorem below.

\begin{thm} \label{thm uniform bounds}

Let $f \in \mathcal{O}_{K}[x]$ be an irreducible monic polynomial of degree $d \geq 3$ with discriminant $\Delta$.  For each $\lambda \in K$, let $J_{\lambda}$ denote the Jacobian of the hyperelliptic curve $C_{\lambda}$ with defining equation $y^{2} = f(x)(x - \lambda)$, and write $G_{2, \lambda} \subseteq \GSp(T_{2}(J_{\lambda}))$ for the image of the associated $2$-adic Galois representation.

a) If $d \geq 4$, the Lie subgroup $G_{2, \lambda} \subseteq \GSp(T_{2}(J_{\lambda}))$ is open for all but finitely many $\lambda \in K$.

b) If $d = 3$ (resp. if $d \geq 5$), then $G_{2, \lambda} \cap \Sp(T_{2}(J_{\lambda})) \supsetneq \Gamma(4)$ for all but finitely many $\lambda \in \mathcal{O}_{K}[(2\Delta)^{-1}] \cdot (K^{\times})^{4}$ (resp. all but finitely many $\lambda \in \mathcal{O}_{K}[(2\Delta)^{-1}] \cdot (K^{\times})^{2}$).

c) If $d = 4$, then we have $G_{2, \lambda} \cap \Sp(T_{2}(J_{\lambda})) \supsetneq \Gamma(16)$ for all but finitely many $\lambda \in K$.  If $d \geq 6$ is even, then we have $G_{2, \lambda} \cap \Sp(T_{2}(J_{\lambda})) \supsetneq \Gamma(4)$ for all but finitely many $\lambda \in K$.

\end{thm}

\begin{proof}

Fix a set of integral ideal representatives for the group of ideal classes of $K$, and let $S$ be the (finite) set of primes of $K$ which divide the ideals in this set.  Let $\Sigma \subset K^{\times}$ denote the multiplicative subgroup generated by the elements $\xi \in \mathcal{O}_{K}$ such that the ideal $(\xi)$ is divisible only by primes which lie in $S$ or divide $(2\Delta)$.  We claim that $\Sigma$ is finitely generated.  Indeed, there is an obvious map from $\Sigma$ to the free $\zz$-module formally generated by the (finite) set of prime ideals of $\mathcal{O}_{K}$ which lie in $S$ or divide $(2\Delta)$, and its kernel is the unit group $\mathcal{O}_{K}^{\times}$, which is also well known to be finitely generated.

Choose any $\lambda \in K$, which we may write as $\mu / \nu$ for some $\mu, \nu \in \mathcal{O}_{K}$ such that the common prime divisors of the ideals $(\mu)$ and $(\nu)$ all lie in $S$.  Let $h(x) = \nu^{d}f(\nu^{-1} x)$, which is a monic polynomial in $\mathcal{O}_{K}[x]$; note that the discriminant of $h$ is equal to $\nu^{d^{2} - d}\Delta$.  Then there is a $K(\sqrt{\nu})$-isomorphism from $C_{\lambda}$ to the hyperelliptic curve $C_{\lambda}'$ whose defining equation is $y^{2} = \nu^{d}f(\nu^{-1}x)(x - \mu) \in \mathcal{O}_{K}[x]$, given by $(x, y) \mapsto (\nu x, \nu^{(d + 1) / 2}y)$.  Thus, letting $J_{\lambda}'$ denote the Jacobian of $C_{\lambda}'$, the $2$-adic Tate modules $T_{2}(J_{\lambda})$ and $T_{2}(J_{\lambda}')$ are isomorphic as $\Gal(\bar{K} / K(\sqrt{\nu}))$-modules.  In light of this, it suffices to replace $G_{2, \lambda}$ with the $2$-adic Galois image associated to $J_{\lambda}'$.

Now Theorem \ref{thm main1} implies that if there is a prime element $\mathfrak{p}$ dividing $(\nu^{d}f(\lambda))$ but not $(2\nu^{d^{2} - d}\Delta)$, then $G_{2, \lambda} \subset \GSp(T_{2}(J_{\lambda})) \cong \GSp(T_{2}(J_{\lambda}'))$ is open.  It follows from the fact that $\mu$ and $\nu$ are coprime away from $S$ that $(\nu^{d}f(\lambda))$ is not divisible by any prime dividing $(\nu)$ which does not lie in $S$, so a prime $\mathfrak{p}$ satisfying the above condition does not divide $(2\Delta)$ and is not in $S$.  The existence of such a prime is equivalent to the condition that $\nu^{d} f(\lambda) \notin \Sigma$, so to prove part (a) it suffices to show that $f(\lambda) \in \Sigma \cdot (K^{\times})^{d}$ for only finitely many $\lambda \in K$.  Note that any such $\lambda$ yields a solution $(x = \lambda, y) \in K \times K$ to an equation of the form $\xi y^{d} = f(x)$ with $\xi \in \Sigma'$, where $\Sigma' \subset \Sigma$ is a set of representatives of elements in $\Sigma / \Sigma^{d}$.  If $d \geq 4$, then an application of the Riemann-Hurwitz formula shows that such an equation defines a smooth curve of genus $\geq 2$, and then Faltings' Theorem implies that there are only finitely many solutions defined over $K$ to each such equation.  Therefore, to prove (a) it suffices to show that there are only finitely many choices of $\xi$.  But this follows from the fact that $\Sigma / \Sigma^{d}$ is finite because $\Sigma$ is finitely generated.

Now assume that $d = 3$ or $d \geq 5$ and that $\lambda \in \mathcal{O}_{K}[(2\Delta)^{-1}] \cdot (K^{\times})^{s}$, with $s = 4$ if $d = 3$ and $s = 2$ otherwise.  Then if we write $\lambda = \mu / \nu$ as above, we have $\nu \in \Sigma \cdot (K^{\times})^{s}$.  Suppose that $\nu^{d} f(\lambda) \notin \Sigma \cdot (K^{\times})^{s}$, which is equivalent to saying that $f(\lambda) \notin \Sigma \cdot (K^{\times})^{s}$.  Then there is a prime element $\mathfrak{p}$ dividing $(\nu^{d} f(\lambda))$ but not $(2\nu^{d^{2} - d}\Delta)$ (so $\mathfrak{p} \nmid (2\Delta)$ and $\mathfrak{p} \notin S$ as before) and such that the maximum integer $m \geq 1$ with $\mathfrak{p}^{m} \mid (\nu^{d}f(\lambda))$ satisfies $v_{2}(m) \leq v_{2}(s) - 1$.  Then Theorem \ref{thm main1} implies that $G_{2, \lambda} \cap \Sp(T_{2}(J_{\lambda})) \supsetneq \Gamma(4)$ both when $d = 3$ and when $d \geq 5$.  Therefore, to prove (b) it suffices to show that $f(\lambda) \in \Sigma \cdot (K^{\times})^{s}$ for only finitely many $\lambda \in K$.  This follows from the same argument as above, on observing by Riemann-Hurwitz that the curves given by $\xi y^{s} = f(x)$ have genus $\geq 2$.

Finally, assume that $d \geq 4$ is even and choose any $\lambda = \mu / \nu \in K$ as before.  Let $s = 4$ if $d = 4$ and let $s = 2$ otherwise.  Suppose that $\nu^{d} f(\lambda) \notin \Sigma \cdot (K^{\times})^{s}$, which in both cases is equivalent to saying that $f(\lambda) \notin \Sigma \cdot (K^{\times})^{s}$.  Then there is a prime element $\mathfrak{p}$ dividing $(\nu^{d} f(\lambda))$ but not $(2\nu^{d^{2} - d}\Delta)$ (so $\mathfrak{p} \nmid (2\Delta)$ and $\mathfrak{p} \notin S$ as before) and such that the maximum integer $m \geq 1$ with $\mathfrak{p}^{m} \mid (\nu^{d}f(\lambda))$ satisfies $v_{2}(m) \leq v_{2}(s) - 1$.  Then Theorem \ref{thm main1} implies that $G_{2, \lambda} \cap \Sp(T_{2}(J_{\lambda}))$ strictly contains $\Gamma(16)$ (resp. $\Gamma(4)$).  Therefore, to prove (c) it suffices to show that $f(\lambda) \in \Sigma \cdot (K^{\times})^{s}$ for only finitely many $\lambda \in K$, which likewise follows from checking that the curves given by $\xi y^{s} = f(x)$ have genus $\geq 2$.

\end{proof}

\begin{rmk} \label{rmk 4-torsion}

Parts (b) and (c) of Theorem \ref{thm uniform bounds} say that for a given polynomial $f$ of degree $d \neq 4$, there are many elements $\lambda \in K$ such that $G_{2, \lambda} \supsetneq \Gamma(4)$.  In these cases, it is always possible to compute the full structure of $G_{2}$ and determine its index in $\GSp(T_{2}(J_{\lambda}))$ by considering the Galois action on the $4$-torsion subgroup of $J$ and using formulas for the generators of the $4$-division field $K(J[4])$ over $K$.  Such formulas are provided by \cite[Proposition 3.1]{yelton2015images} in the case that $d$ is odd and are found in \cite[\S2.4]{yelton2015hyperelliptic} in the case that $d$ is even.

\end{rmk}

\section{An alternate approach to producing the transvections} \label{S6}

In this final section, we briefly discuss another method of proving Proposition \ref{prop Galois action local} in the $\ell = 2$ case, suggested by one of the referees, which involves N\'{e}ron models of Jacobians with semi-abelian reduction rather than Galois actions on \'{e}tale fundemantal groups of punctured projective lines.

We assume all of the notation of \S\ref{S2}, in particular that we have a strictly Henselian DVR $\mathcal{R}_{\mathfrak{p}}$ with uniformizer $\pi$ and are considering the hyperelliptic curve $C$ defined over its field of fractions $\mathcal{K}_{\mathfrak{p}}$; to simplify notation, we write $\mathcal{R}$ and $\mathcal{K}$ for $\mathcal{R}_{\mathfrak{p}}$ and $\mathcal{K}_{\mathfrak{p}}$ respectively, and we write $k$ for the residue field $\mathcal{R} / \mathfrak{p}$.  As in \S\ref{S2}, our curve $C$ is defined by the equation $y^2 = \prod_{i = 1}^d (x - \alpha_i)$, where the roots $\alpha_i$ all lie in $\mathcal{R}$; for $1 \leq i \leq d$, let $\bar{\alpha}_i \in k$ denote the image modulo $\mathfrak{p}$ of $\alpha_i$.  We assume that the elements $\bar{\alpha}_i \in k$ are all distinct except that $\bar{\alpha}_1 = \bar{\alpha}_2$ and let $m \geq 1$ be the greatest integer such that $\pi^m | (\alpha_1 - \alpha_2)$; this is precisely the main hypothesis of Proposition \ref{prop Galois action local}, except that we are assuming without loss of generality that $\alpha_1$ and $\alpha_2$ are the two roots which coalesce in the reduction.  If $d = 2g + 1$, the curve has a single point at infinity which we denote by $\infty \in C(\mathcal{K})$; if $d = 2g + 2$, the curve has two points at infinity which we denote by $\infty_1, \infty_2 \in C(\mathcal{K})$.

The Jacobian $J / \mathcal{K}$ of $C$ admits a \textit{N\'{e}ron model} $\mathcal{J} / \mathcal{R}$, which is a smooth, commutative group scheme.  Writing  and $\mathcal{J}_k / k$ for the special fiber of $\mathcal{J}$, we see that $\mathcal{J}_k$ is a (not necessarily connected) commutative group variety over $k$; we denote the connected component containing the identity by $\mathcal{J}_k^0$.  The hypothesis that only the roots $\alpha_1$ and $\alpha_2$ are equivalent modulo $\mathfrak{p}$ implies that $C$ has semistable reduction at $\mathfrak{p}$.  Therefore, the Jacobian $J / \mathcal{K}$ has semiabelian reduction by \cite[Corollary 9.7.2]{bosch2012neron}, which means that $\mathcal{J}_k^0$ decomposes as an extension of an abelian variety by a torus $\mathcal{T} \subseteq \mathcal{J}_k^0$.

\begin{lemma} \label{lemma component group}

The group of connected components $\mathcal{J}_k / \mathcal{J}_k^0$ is isomorphic to $\zz / 2m\zz$.

\end{lemma}

\begin{proof}

Let $\mathcal{C}' / \mathcal{R}$ be the projective curve defined over $\mathcal{R}$ by the same equations defining $C / \mathcal{K}$.  It is easy to verify that the special fiber of $\mathcal{C}'$ has a node of thickness $2m$ at the point $(\bar{\alpha}_i, 0)$, where $\bar{\alpha}_i = \bar{\alpha}_j$ is the image in $k$ of $\alpha_i, \alpha_j \in \mathcal{R}$.  Let $\mathcal{C} / \mathcal{R}$ be the minimal desingularization of $\mathcal{C}'$.  By \cite[Corollary 10.3.25]{liu2002algebraic}, the special fiber of $\mathcal{C}$ has $2m$ irreducible components, configured in such a way that its \textit{dual graph} $\Gamma(\mathcal{C}_k)$ (whose vertices correspond to irreducible components and whose edges correspond to intersections between two irreducible components) consists of a single cycle of length $2m$.  It then follows from \cite[Remark 9.6.12]{bosch2012neron} that $\mathcal{J}_k / \mathcal{J}_k^0 \cong \zz / 2m\zz$.

\end{proof}

Let $T_2(\mathcal{J}_k)$ denote the $2$-adic Tate module of the commutative algebraic group $\mathcal{J}_k$; we identify it with the submodule $T_2(J)^{G_{\mathcal{K}}} \subset T_2(J)$ of invariants under the absolute Galois group $G_{\mathcal{K}}$ of $\mathcal{K}$ (to see this identification, recall that since $\mathcal{R}$ is strictly Henselian, the absolute Galois group of $\mathcal{K}$ coincides with its inertia subgroup, and use \cite[Proposition 2.2.5]{grothendieck1972modeles}).  Letting $T_2(\mathcal{T})$ be the $2$-adic Tate module of the torus $\mathcal{T}$, we now have a filtration $T_2(\mathcal{T}) \subseteq T_2(J)^{G_{\mathcal{K}}} \subset T_2(J)$.

\begin{lemma} \label{lemma toral part}

The submodule $T_2(\mathcal{T}) \subset T_2(J)$ has rank $1$ and is generated over $\zz_2$ by an element whose image in $T_2(J) / 2T_2(J) = J[2]$ is the class of the divisor $D := (\alpha_i, 0) + (\alpha_j, 0) - 2(\infty) \in \Div^0(C)$ (resp. $(\alpha_i, 0) + (\alpha_j, 0) - (\infty_1) - (\infty_2) \in \Div^0(C)$) if $d = 2g + 1$ (resp. if $d = 2g + 2$).

\end{lemma}

\begin{proof}

By \cite[Corollary 9.7.2]{bosch2012neron}, the identity component $\mathcal{J}^0$ (whose special fiber is $\mathcal{J}_k^0$ and whose generic fiber is $J$) can be identified with the degree-$0$ Picard scheme $\Pic^0(\mathcal{C}) / \mathcal{R}$ of the relative curve $\mathcal{C} / \mathcal{R}$.  In particular, the special fibers $\mathcal{J}_k^0$ and $\Pic^0(\mathcal{C}_k)$ are isomorphic.  Let $\nu : \widetilde{\mathcal{C}}_k \to \mathcal{C}_k$ be a normalization of $\mathcal{C}_k$, and let $\nu^{*} : \Pic^0(\mathcal{C}_k) \to \Pic^0(\widetilde{\mathcal{C}}_k)$ be the induced pullback map.  Now \cite[Example 9.2.8]{bosch2012neron} implies that the torus $\mathcal{T} \subseteq \mathcal{J}_k^0$ is identified under the isomorphism $\mathcal{J}_k^0 \cong \Pic^0(\mathcal{C}_k)$ with the kernel of $\nu^{*}$, and that moreover, $\mathcal{T}$ has rank equal to the rank of the singular cohomology group $H^1(\Gamma(\mathcal{C}_k), \zz)$, where $\Gamma(\mathcal{C}_k)$ is the dual graph as in the proof of Lemma \ref{lemma component group}.  It was shown in that proof that $\Gamma(\mathcal{C}_k)$ consists of a single cycle, so this rank is $1$ and the first claim is proved.

Now to prove the second claim, it is enough to show that the divisor class $[D] \in \Pic^0(C)$ extends to a divisor class in $\Pic^0(\mathcal{C})$ whose restriction to $\Pic^0(\mathcal{C}_k)$ lies in the kernel of $\nu^{*}$.  If $d = 2g + 1$ (resp. if $d = 2g + 2$), then the divisor $D$ is linearly equivalent to $D' := \sum_{3 \leq i \leq 2g + 1} (\alpha_i, 0) - (2g - 1)(\infty) \in \Div^0(C)$ (resp. $D := \sum_{3 \leq i \leq 2g + 2} (\alpha_i, 0) - g(\infty_1) - g(\infty_2) \in \Div^0(C)$) if $d = 2g + 1$; thus, we have $[D] = [D']$.  Now it is clear that the divisor $D'$ extends to a divisor of $\mathcal{C}$ over $\mathcal{R}$; we denote its restriction to $\mathcal{C}_k$ also by $D'$.  Moreover, since the normalization $\widetilde{\mathcal{C}}_k$ can be defined by the equation $y^2 = \prod_{3 \leq i \leq d} (x - \bar{\alpha}_i)$ with $\nu$ taking $(x, y)$ to $(x, y(x - \bar{\alpha}_i)^{-1})$, it is clear that the pullback of $D$ with respect to $\nu$ is the divisor of the function $y \in k(\widetilde{\mathcal{C}}_k)$ and therefore $\nu^{*}([D']) = 0 \in \Pic^0(\widetilde{\mathcal{C}}_k)$.  Since $2D \in \Div^{0}(\mathcal{C}_k)$ is the divisor of the function $(x - \bar{\alpha}_1)(x - \bar{\alpha}_2) \in \Div^0(\mathcal{C}_k)$ while an application of the Riemann-Roch formula shows that $D$ itself is not principal, we see that $[D]$ generates the kernel of $\nu^{*}$ restricted to the $2$-torsion subgroup of $\Pic^0(\mathcal{C}_k)$, and it follows that $T_2(\mathcal{T})$ is generated by an element whose image modulo $2$ is the class of $D$.

\end{proof}

The statement of the following proposition essentially implies that of Proposition \ref{prop Galois action local} for the $\ell = 2$ case, given the well-known equivalence between the algebraic and topological characterizations of the $2$-torsion subgroup of $\Pic^0(C)$ (as discussed in the footnote to the proof of \cite[Lemma 8.12]{mumford1984tata}).

\begin{prop} \label{prop Galois action local2}

With all of the above notations and assumptions, the image of the natural action of $G_{\mathcal{K}}$ on $T_2(J)$ is topologically generated by the element $t_{c}^{2m} \in \Sp(T_2(J))$ for some $c \in T_2(J)$ whose image in $T_2(J) / 2T_2(J) = J[2]$ is the class of the divisor $(\alpha_1, 0) + (\alpha_2, 0) - 2(\infty) \in \Div^0(C)$ (resp. $(\alpha_1, 0) + (\alpha_2, 0) - (\infty_1) - (\infty_2) \in \Div^0(C)$) if $d = 2g + 1$ (resp. if $d = 2g + 2$).

\end{prop}

\begin{proof}

Grothendieck's Orthogonality Theorem \cite[Th\'{e}or\`{e}me 2.4]{grothendieck1972modeles} says that the $\zz_2$-submodules $T_2(\mathcal{T}), T_2(J)^{G_{\mathcal{K}}} \subset T_2(J)$ are orthogonal to each other under the $2$-adic Weil pairing.  It easily follows from this and from the fact that the action of $G_{\mathcal{K}}$ respects the Weil pairing that for any $\sigma \in G_{\mathcal{K}}$, the operator $\rho_2(\sigma) - 1 \in \End(T_2(J))$ has kernel containing $T_2(J)^{G_{\mathcal{K}}}$ and image contained in $T_2(\mathcal{T})$ so that this operator induces a homomorphism $T_2(J) / T_2(J)^{G_{\mathcal{K}}} \to T_2(\mathcal{T})$; in fact, this yields a homomorphism 
$$N : G_{\mathcal{K}} \to \Hom_{\zz_2}(T_2(J) / T_2(J)^{G_{\mathcal{K}}}, T_2(\mathcal{T}))$$
 given by sending an element $\sigma \in G_{\mathcal{K}}$ to the homomorphism $T_2(J) / T_2(J)^{G_{\mathcal{K}}} \to T_2(\mathcal{T})$ induced by $\rho_2(\sigma) - 1$.  Since the above group of $\zz_2$-module homomorphisms is clearly itself a pro-$2$ group, the homomorphism $N$ factors through the maximal pro-$2$ (tame) quotient $G_{\mathcal{K}}^{(2)} \cong \zz_2$ of $G_{\mathcal{K}}$.  The image of a generator of $G_{\mathcal{K}}^{(2)}$ under $N$ can be interpreted in terms of the so-called \textit{monodromy pairing} associated to $\mathcal{T}$ -- see \cite[\S3.2]{papikian2013non} for a concise explanation.  Now \cite[Th\'{e}or\`{e}me 11.5]{grothendieck1972modeles} implies that the image under $N$ of a generator of $G_{\mathcal{K}}^{(2)}$, viewed as a homomorphism in $\Hom_{\zz_2}(T_2(J) / T_2(J)^{G_{\mathcal{K}}}$, has cokernel isomorphic to the $2$-power part of $\mathcal{J}_k / \mathcal{J}_k^0$.  Lemma \ref{lemma component group} implies that this cokernel is isomorphic to $\zz / 2^{v_2(2m)}\zz$.  It follows from the above set-up that for some $c \in T_2(J)$ which generates the rank-$1$ submodule $T_2(\mathcal{T})$, the operator $t_c^{2^{v_2(2m)}} \in \Sp(T_2(J))$ lies in the $2$-adic Galois image associated to $J$; by taking an apporopriate power of this operator, we get that $t_c^{2m}$ lies in this Galois image.  Now Lemma \ref{lemma toral part} implies that the image of $c$ modulo $2$ is identified with the divisor class given in the statement of the proposition, and the proposition is proved.

\end{proof}

\bibliographystyle{plain}
\bibliography{bibfile}

\end{document}